\documentclass[table,svgnames,english,oneside]{smfart}
\newcommand{\opn}[1]{\operatorname{#1}}

\usepackage{amsthm,mathtools,amssymb,textcomp,extarrows,bm,mleftright,graphicx,stmaryrd,scalerel}
\mleftright

\makeatletter
\AtBeginDocument{%
 \let\glb@currsize\relax
}
\makeatother

\usepackage{mathrsfs}
\usepackage{xparse}
\ExplSyntaxOn
\NewDocumentCommand{\definealphabet}{mmmm}
 {%
  \int_step_inline:nnn { `#3 } { `#4 }
   {
    \cs_new_protected:cpx { #1 \char_generate:nn { ##1 }{ 11 } }
     {
      \exp_not:N #2 { \char_generate:nn { ##1 } { 11 } }
     }
   }
 }
 \makeatletter\providecommand*\xLongrightleftharpoons[2][]{\mathrel{%
  \raise.22ex\hbox{%
    $\ext@arrow 9999\MT_rightharpoonup_fill:{\phantom{\Longleftrightarrow}}{#2}$}%
  \setbox0=\hbox{%
    $\ext@arrow 9999\MT_leftharpoondown_fill:{#1}{\phantom{\Longleftrightarrow}}$}%
  \kern-\wd0 \lower.22ex\box0}}
\ExplSyntaxOff
\definealphabet{bb}{\mathbb}{A}{Z}
\definealphabet{cal}{\mathcal}{A}{Z}
\definealphabet{frak}{\mathfrak}{A}{z}
\definealphabet{rm}{\mathrm}{A}{z}
\definealphabet{bf}{\mathbf}{A}{Z}
\definealphabet{scr}{\mathscr}{A}{Z}
\definealphabet{tt}{\mathtt}{A}{Z}
\DeclareMathOperator{\Gal}{Gal}
\DeclareMathOperator{\FT}{\frakF}
\DeclareMathOperator{\sep}{sep}
\DeclareMathOperator{\Rep}{\mathbf{Rep}}

\newcommand{\lto}{\longrightarrow}
\newcommand{\wtilde}[1]{\widetilde{#1}}

\newcommand{\repcat}{\Rep_{\bfZ_p}(\scrG_K)}
\newcommand{\Qp}{{\bfQ_p}}
\newcommand{\Zp}{\bfZ_p}

\newcommand{\Cpflat}{\bfC_p^\flat}

\usepackage{tikz}
\usetikzlibrary{decorations.pathmorphing,cd,arrows,calc,fit,intersections,decorations.pathreplacing,fadings,patterns}

\usepackage[new]{old-arrows}

\newcommand{\hyphen}{\!\opn{-}\!}

\usepackage{calc}
\makeatletter
\newcommand*{\textvcenter}{\@ifstar{\@tempswatrue\text@vcenter}{\@tempswafalse\text@vcenter}}
\newcommand*{\text@vcenter}[1]{\mbox{$\m@th\vcenter{\setbox\z@=\hbox{#1}\if@tempswa\dp\z@\z@\fi\box\z@}$}}
\makeatother

\newlength{\owidth}
\newcommand{\pparen}[1]{
\mathopen{\text{\textvcenter{\includegraphics[width=\owidth,height=\totalheightof{$\left(#1\right)$}]{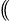}}}}#1\mathclose{\text{\textvcenter{\includegraphics[width=\owidth,height=\totalheightof{$\left(#1\right)$}]{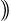}}}}}

\newcommand{\bbrac}[1]{\left\llbracket #1\right\rrbracket}

\newcommand{\np}{\opn{np}}

\newcommand{\EtmKnpc}{\bfE_{\FT,K}^{\np,\circ}}
\newcommand{\EtmKnp}{\bfE_{\FT,K}^{\np}}
\newcommand{\EtmKcnp}{\widehat{\bfE}_{\FT,K}^{\np,\circ}}
\newcommand{\Esep}{\bfE_{\FT}^{\np}}
\newcommand{\Ecsep}{\widehat{\bfE}_{\FT}^{\np}}

\usepackage{fontawesome}

\newlength{\ltolen}
\DeclareSymbolFont{yhlargesymbols}{OMX}{yhex}{m}{n}
\DeclareMathAccent{\yhwidehat}{\mathord}{yhlargesymbols}{"62}

\newcommand{\AKq}{\bfA_{\frakF,K}^{\np,?}}

\usepackage{quiver}
\usepackage[pdfencoding=auto,psdextra]{hyperref}
\usepackage[nameinlink]{cleveref}
\Crefname{theorem}{Theorem}{Theorem}
\Crefname{conjecture}{Conjecture}{Conjectures}
\Crefname{lemma}{Lemma}{Lemmas}
\Crefname{definition}{Definition}{Definitions}
\Crefname{remark}{Remark}{Remarks}
\Crefname{proposition}{Proposition}{Propositions}
\Crefname{corollary}{Corollary}{Corollaries}
\Crefname{equation}{}{}
\Crefname{item}{}{}
\Crefname{example}{Example}{Examples}
\Crefname{proof}{Proof}{Proofs}
\Crefname{condition}{Condition}{Conditions}
\Crefname{question}{Question}{Questions}
\makeatletter
\renewcommand*\subjclass[2][2020]{%
  \def\@subjclass{#2}%
  \@ifundefined{subjclassname@#1}{%
    \ClassWarning{\@classname}{Unknown edition (#1) of Mathematics
      Subject Classification; using '2020'.}%
  }{%
    \@xp\let\@xp\subjclassname\csname subjclassname@#1\endcsname
  }%
}
\@namedef{subjclassname@1991}{%
  \textup{1991} Mathematics Subject Classification}
\@namedef{subjclassname@2000}{%
  \textup{2000} Mathematics Subject Classification}
\@namedef{subjclassname@2010}{%
  \textup{2010} Mathematics Subject Classification}
\@namedef{subjclassname@2020}{%
  \textup{2020} Mathematics Subject Classification}
\@xp\let\@xp\subjclassname\csname subjclassname@2020\endcsname
\makeatother
\usepackage[inline]{enumitem}
\usepackage{float,algorithm,algpseudocode,cases,adjustbox,booktabs,caption,makecell,csquotes}
\newtheorem{theorem}{Theorem}[section]
\newtheorem{example}[theorem]{Example}
\AddToHook{env/example/begin}{\crefalias{theorem}{example}}
\newtheorem{lemma}[theorem]{Lemma}
\AddToHook{env/lemma/begin}{\crefalias{theorem}{lemma}}
\newtheorem{remark}[theorem]{Remark}
\AddToHook{env/remark/begin}{\crefalias{theorem}{remark}}
\newtheorem{proposition}[theorem]{Proposition}
\AddToHook{env/proposition/begin}{\crefalias{theorem}{proposition}}
\newtheorem{definition}[theorem]{Definition}
\AddToHook{env/definition/begin}{\crefalias{theorem}{definition}}

\AddToHook{env/conjecture/begin}{\crefalias{theorem}{conjecture}}
\newtheorem{corollary}[theorem]{Corollary}
\AddToHook{env/corollary/begin}{\crefalias{theorem}{corollary}}
\newtheorem{question}[theorem]{Question}
\AddToHook{env/question/begin}{\crefalias{theorem}{question}}
\newtheorem{assumption}[theorem]{Assumption}
\AddToHook{env/assumption/begin}{\crefalias{theorem}{assumption}}

\newtheorem{introthm}{Theorem}

\numberwithin{equation}{section}
\numberwithin{figure}{section}
\newlist{propenum}{enumerate}{1}
\setlist[propenum]{label=(\arabic*), ref=\theproposition~(\arabic*)}
\crefalias{propenumi}{proposition}
\newlist{lemenum}{enumerate}{1}
\setlist[lemenum]{label=(\arabic*), ref=\thelemma~(\arabic*)}
\crefalias{lemenumi}{lemma}
\newlist{thmenum}{enumerate}{1}
\setlist[thmenum]{label=(\arabic*), ref=\thetheorem~(\arabic*)}
\crefalias{thmenumi}{theorem}
\Crefformat{enumi}{#2\textup{(#1)}#3}
\usepackage{microtype}

\title{Multivariable period rings of $p$-adic false Tate curve extension}

\usepackage{orcidlink}
\author{Yijun Yuan\orcidlink{0000-0001-6571-6980}}
\address{Institute for Theoretical Sciences, Westlake University, No. 600 Dunyu Road, Sandun town, Xihu district, Hangzhou, Zhejiang Province, 310030, China}
\email{941201yuan@gmail.com}
\urladdr{https://yijunyuan.github.io/}

\begin{document}
\frontmatter
\begin{abstract}
	Let $p\geq 3$ be a prime number and $K$ be a finite extension of $\bfQ_p$ with uniformizer $\pi_K$. In this article, we introduce two multivariable period rings $\bfA_{\frakF,K}^{\np}$ and $\bfA_{\frakF,K}^{\np,\opn{c}}$ for the \'etale $(\varphi,\Gamma_{\frakF,K})$-modules of $p$-adic false Tate curve extension $K\left(\pi_K^{1/p^\infty},\zeta_{p^\infty}\right)$. Various properties of these rings are studied and as applications, we show that $(\varphi,\Gamma_{\frakF,K})$-modules over these rings bridge $(\varphi,\Gamma)$-modules and $(\varphi,\tau)$-modules over imperfect period rings in both classical and cohomological sense, which answers a question of Caruso. Finally, we construct the $\psi$ operator for false Tate curve extension and discuss the possibility to calculate Iwasawa cohomology for this extension via $(\varphi,\Gamma_{\frakF,K})$-modules over these rings.
\end{abstract}
\maketitle
\tableofcontents
\mainmatter
\section*{Notations and conventions}
For any field $F$, denote by $\scrG_F$ the absolute Galois group of $F$.

Let $p\geq 3$ be a prime number. We normalize the $p$-adic valuation $v_p$ on the field $\bfC_p$ of $p$-adic complex numbers to be $v_p(p)=1$. Denote by $v^\flat$ the valuation on $\Cpflat$ and let $\bfA_{\opn{inf}}\coloneqq W\left(\scrO_{\bfC_p}^\flat\right)$ be Fontaine's infinitesimal period ring.

Let $K$ be a finite extension of $\Qp$, with uniformizer $\pi_K$ and residue field $k_K$. Let $\repcat$ be the category of $\bfZ_p$-representations of $\scrG_K$.

\section{Introduction}
In 1990, Fontaine attached to every $\Zp$-representation $V$ of $\scrG_K$ an \'etale $(\varphi,\Gamma)$-module $\mathbf{D}(V)$, which is a module over the $p$-Cohen ring (cf. \Cref{def:34646}) $\bfA_K$ of the field of norms (cf. \Cref{def:69346}) $X_K(K_{p^\infty})$ of the cyclotomic extension $K_{p^\infty}\coloneqq K(\zeta_{p^\infty})$ over $K$, equipped with a semi-linear action of the Frobenius map $\varphi$ and $\Gamma_K\coloneqq \Gal(K_{p^\infty}/K)$ that commutes with each other.
The functor $V\longmapsto \bfD(V)$ induces the equivalence of categories between $\repcat$ and \'etale $(\varphi,\Gamma)$-modules over $\mathbf{A}_K$ (cf. \Cref{thm:36449}). This turns out to be a huge success in $p$-adic Hodge theory and has numerous important applications. Among them, we can use \'etale $(\varphi,\Gamma)$-module $\bfD(V)$ to calculate the Iwasawa cohomology of $V$:
\begin{theorem}[{cf. \cite[Théorème II.1.3]{cherbonnierTheorieIwasawaRepresentations1999}}]\label{thm:49557}
	For $V\in\repcat$, define the Iwasawa cohomology of $V$ to be $$\rmH_{\opn{Iw}}^i\left(\scrG_K,V\right)\coloneqq \rmH^i\left(\scrG_K,\bfZ\bbrac{\Gamma_K}\otimes_{\Zp}V\right).$$
	Then one has $\rmH_{\opn{Iw}}^1\left(\scrG_K,V\right)\cong \bfD(V)^{\psi=1}$, where $\psi\colon \bfD(V)\lto\bfD(V)$ is the unique operator that commutes with the action of $\Gamma_K$ and satisfies $\psi(\varphi(a)\cdot x)=a\cdot \psi(x)$ and $\psi(a\cdot\varphi(x))=\psi(a)\cdot x$ for any $a\in\bfA_K$ and $x\in \bfD(V)$.
\end{theorem}

This is a crucial ingredient in Colmez's proof of the $p$-adic local Langlands correspondence for $\opn{GL}_2(\Qp)$ and has other applications in Iwasawa theory.

In 2004, J. Coates, T. Fukaya, K. Kato, R. Sujatha, and O. Venjakob proposed a program of non-commutative Iwasawa theory (cf. \cite{coatesGL2MainConjecture2005}). In view of the influential role played by the theory of $(\varphi, \Gamma)$-modules in commutative Iwasawa theory, it is natural to ask if there is an analogy of the \Cref{thm:49557} in the non-commutative situation. The first interesting case can be the false Tate curve extension $K_{\FT}\coloneqq K\left(\pi_K^{1/p^\infty},\zeta_{p^\infty}\right)$, which is the normal closure of the Kummer extension $K\left(\pi_K^{1/p^\infty}\right)$ over $K$. To imitate what has been done in the cyclotomic extension, the first step is to lift the action of the Frobenius map $\varphi$ and the Galois group $\Gamma_{\FT,K}\coloneqq \Gal(K_{\FT}/K)$, which is a $p$-adic Lie group of dimension $2$, on the field of norms $X_K\left(K_{\FT}\right)$ to its $p$-Cohen ring. This turns out to be an extremely difficult problem, for it is open to find a norm-compatible system of uniformizers for the tower $\left\{K\left(\pi_K^{1/p^n},\zeta_{p^n}\right)\right\}_{n\geq 0}$ even when $K=\Qp$. Although there are some partial results in this direction (cf. \cite{vivianiRamificationGroupsArtin2004,bellemareExplicitUniformizersCertain2020a,WangYuan2021,wangUniformizerFalseTate2024}), the general case is still open. On the other hand, requiring the $\psi$ operator to be non-trivial forces us to find an imperfect period ring, which rules out the perfectoid field $\wtilde{\bfE}_{\frakF,K}\coloneqq \widehat{K}_{\frakF}^\flat$. %

When constructing the $(\varphi,\tau)$-modules over ``partially imperfect'' (partiellement déperfectisées) period rings, Caruso considered the field $\bfE_{\FT,K}^{\mathtt{Car},\np ,\circ}$, which is the fraction field of the image of the injective morphism (cf. \cite[Proposition 1.7]{Caruso2013})
$$\iota_{\mathtt{Car}}\colon k_K\bbrac{X,Y}\lto \scrO_{\bfC_p}^\flat,\ X\longmapsto u_K,\ Y\longmapsto \eta_{\Qp}\coloneqq\varepsilon-1,$$
where $u_K\coloneqq \left(\pi_K,\pi_K^{1/p},\cdots,\pi_K^{1/p^n},\cdots\right)\in\Cpflat$  and $\varepsilon\coloneqq\left(1,\zeta_p,\cdots,\zeta_{p^n},\cdots\right)\in\Cpflat$.

The action of $\scrG_K$ on $\bfE_{\FT,K}^{\mathtt{Car},\np ,\circ}$ factors through $\Gamma_{\frakF,K}\coloneqq \Gal\left(K_{\frakF}/K\right)$, which allows Caruso to define the following variants of imperfect period rings:
$$\bfE_{\FT,K}^{\mathtt{Car},\np }\coloneqq \left(\left(\bfE_{\FT,K}^{\mathtt{Car},\np ,\circ}\right)^{\sep}\right)^{\scrG_{K_{\frakF}}},\ \bfE_{\FT,K}^{\mathtt{Car},u\hyphen\np }\coloneqq\left(\left(\bigcup_n\bfE_{\FT,K}^{\mathtt{Car},\np ,\circ}\left(\eta_{\Qp}^{1/p^n}\right)\right)^{\sep}\right)^{\scrG_{K_{\frakF}}}$$
and
$$\bfE_{\FT,K}^{\mathtt{Car},\eta\hyphen\np }\coloneqq\left(\left(\bigcup_n\bfE_{\FT,K}^{\mathtt{Car},\np ,\circ}\left[u_K^{1/p^n}\right]\right)^{\sep}\right)^{\scrG_{K_{\frakF}}}.$$
It is shown by Caruso that the absolute Galois groups of these fields are all isomorphic to $H_{\FT,K}$, and the action of $\Gamma_{\FT,K}$ can be lifted to their $p$-Cohen rings respectively.

This inspires us to study the multivariable period rings for false Tate curve extension. For simplicity, we make the following mild assumption:
\begin{assumption}
	The false Tate curve extension $K_{\frakF}$ is totally ramified over $K$. In particular, the residue field of $X_K(K_{\frakF})$ and $X_K(K_{p^\infty})$ are both $k_K$.
\end{assumption}
For example, this assumption is satisfied when $K=\Qp$ (cf. \cite[Theorem 5.5]{vivianiRamificationGroupsArtin2004}).

\begin{definition}[cf. \Cref{def:9888}, \Cref{def:8032}]
	Let $\bfE_{\FT,K}^{\np ,\circ,+}$ be the image of the injective (cf. \Cref{lem:1122}) morphism
	$$\iota_{\np}\colon k_K\bbrac{X,Y}\lto \scrO_{\bfC_p}^\flat, f\longmapsto f(u_K,\eta_K),$$
	where $\eta_K$ is a uniformizer of $\bfE_K\coloneqq X_K\left(K_{p^\infty}\right)$. Let $\bfE_{\FT,K}^{\np ,\circ}$ be the fraction field of $\bfE_{\FT,K}^{\np ,\circ,+}$, $\bfE_{\FT,K}^{\np ,\circ,sep}$ be the separable closure of $\bfE_{\FT,K}^{\np ,\circ}$ in $\bfC_p^\flat$ and $\bfE_{\FT,K}^{\np }\coloneqq \left(\bfE_{\FT,K}^{\np ,\circ,sep}\right)^{\scrG_{K_{\frakF}}}$. Besides that, let $\widehat{\bfE}_{\FT,K}^{\np ,\circ}$ be the completion of $\bfE_{\FT,K}^{\np ,\circ}$ with respect to the valuation $v^\flat$.
\end{definition}
\begin{remark}
	When $K=\Qp$, one can take $\eta_{\Qp}=\varepsilon-1$, and consequently $\bfE_{\FT,\Qp}^{\np ,\circ}$ (resp. $\bfE_{\FT,\Qp}^{\np }$) coincides with the field $\bfE_{\FT,\Qp}^{\mathtt{Car},\np ,\circ}$ (resp. $\bfE_{\FT,\Qp}^{\mathtt{Car},\np }$) defined by Caruso.
\end{remark}
Our first main result is the clarification of properties of the above fields (and ring $\bfE_{\FT,K}^{\np,\circ,+}$):
\begin{introthm}\label{thm:42581}
	One has strict containment $\bfE_{\FT,K}^{\np ,\circ}\subsetneq \bfE_{\FT,K}^{\np }\subsetneq \widehat{\bfE}_{\FT,K}^{\np ,\circ}\subsetneq \wtilde{\bfE}_{\FT,K}$. Moreover, we summarize the properties of these fields (and the ring $\bfE_{\FT,K}^{\np,\circ,+}$) in the following table:
	\begin{center}
		\begin{tabular}{ccccccc}
			\toprule
			$E$                                            & $\scrG_{E}\cong \scrG_{K_{\frakF}}$?       & $\substack{(E,v^\flat)                                                                                                                                                                                          \\\text{is complete?}}$                 & $\substack{v^\flat\vert_E\\\text{is discrete? }}$             & $\substack{\Gamma_{\FT,K}\text{-action}                                                                               \\\text{is liftable?}}$ & $\substack{E^{\scrG_{K_{p^\infty}}}=\bfE_K,                                \\E^{\scrG_{K_\infty}}=\bfE_{\tau,K}}$ & $[E\colon \varphi(E)]$ \\
			\midrule
			\rowcolor{Turquoise!30}$\bfE_{\FT,K}$          & {\normalfont\faCheck}                      & {\normalfont\faCheck}                      & {\normalfont\faCheck}                      & {\normalfont\faQuestion}                   & {\normalfont\faTimes}                     & $p$                          \\
			$\bfE_{\FT,K}^{\np ,\circ,+}$                  & {\normalfont\faMinus}                      & \hyperref[lem:2016]{\normalfont\faCheck}   & \hyperref[prop:35357]{\normalfont\faTimes} & {\normalfont\faMinus}                      & {\normalfont\faMinus}                     & {\normalfont\faMinus}        \\
			$\EtmKnpc$                                     & \hyperref[prop:56478]{\normalfont\faTimes} & \hyperref[prop:51851]{\normalfont\faTimes} & \hyperref[prop:35357]{\normalfont\faTimes} & \hyperref[prop:42040]{\normalfont\faCheck} & \hyperref[lem:56945]{\normalfont\faCheck} & \hyperref[prop:16686]{$p^2$} \\
			$\EtmKnp$                                      & \hyperref[lem:5466]{\normalfont\faCheck}   & \hyperref[prop:29251]{\normalfont\faTimes} & \hyperref[prop:35357]{\normalfont\faTimes} & \hyperref[prop:42040]{\normalfont\faCheck} & \hyperref[lem:56945]{\normalfont\faCheck} & \hyperref[prop:16686]{$p^2$} \\
			$\EtmKcnp$                                     & \hyperref[lem:47550]{\normalfont\faCheck}  & {\normalfont\faCheck}                      & \hyperref[prop:35357]{\normalfont\faTimes} & \hyperref[prop:42040]{\normalfont\faCheck} & \hyperref[lem:56945]{\normalfont\faCheck} & \hyperref[prop:16686]{$p^2$} \\
			\rowcolor{Turquoise!30}$\wtilde{\bfE}_{\FT,K}$ & {\normalfont\faCheck}                      & \normalfont\faCheck                        & \normalfont\faTimes                        & {\normalfont\faCheck}                      & \normalfont\faTimes                       & $1$                          \\
			\bottomrule
		\end{tabular}
	\end{center}
	where \normalfont\faMinus\space means ``not applicable'', \normalfont\faCheck\space means ``true'', \normalfont\faTimes\space means ``false'' and \normalfont\faQuestion\space means ``unknown''.

\end{introthm}
\begin{remark}
	The results in the colored rows are not original of this paper and can be found in the literature (for example, \cite{Wintenberger1983,scholzePerfectoidSpaces2012}). We put them here to make a comparison with our results. The results in the white rows can be found in \Cref{lem:5466}, \Cref{prop:51851}, \Cref{lem:2016}, \Cref{prop:56478}, \Cref{lem:47550}, \Cref{prop:29251}, \Cref{prop:35357}, \Cref{prop:42040}, \Cref{lem:56945} and \Cref{prop:16686}.
\end{remark}
\begin{definition}
	Let $\bfA_{\FT,\Qp}^{\np}$ (resp. $\bfA_{\FT,\Qp}^{\np,\opn{c}}$) be the $p$-Cohen ring of $\bfE_{\FT,\Qp}^{\np}$ (resp. $\widehat{\bfE}_{\FT,\Qp}^{\np}$) in $W\left(\wtilde{\bfE}_{\frakF,K}\right)$ that equipped with the lifted action of Frobenius and $\Gamma_{\frakF,K}$ from the residue field.
\end{definition}
\begin{remark}
	The rings $\bfA_{\FT,K}^{\np}$ and $\bfA_{\FT,K}^{\np,\rmc}$ in this article are very artificially constructed. It is valuable to seek a natural interpretation within the framework of (mixed-characteristic) locally analytic vectors (cf. \cite{porat2024locallyanalyticvectorsdecompletion}), as Gal Porat suggested to the author. 
\end{remark}
We give several applications of this fundamental theorem:

\subsection*{Bridging $(\varphi,\Gamma)$-modules and $(\varphi,\tau)$-modules}
The theory of $(\varphi,\tau)$-modules is the analogue of $(\varphi,\Gamma)$-modules for the Kummer extension constructed by Caruso in 2013. This variant is naturally connected to the Breuil-Kisin modules, which plays an important role in integral $p$-adic Hodge theory. In \cite[Section 4]{Caruso2013}, Caruso asked how to connect $(\varphi,\Gamma)$-modules and $(\varphi,\tau)$-modules (cf. \Cref{eg:5636}) without the bridging of the category of Galois representations. In his doctoral thesis, Poyeton partially answered this question by using the theory of locally (resp. pro-) analytic vectors (cf. \cite[Théorème I]{poyetonExtensionsLiePadiques2019}). The following theorem provides a fresh solution to this question in terms of \'etale $(\varphi,\Gamma_{\frakF,K})$-modules (cf. \Cref{eg:5636}) over multivariable period rings $\bfA_{\frakF,K}^{\np}$ and $\bfA_{\frakF,K}^{\np,\rmc}$:
\begin{introthm}[{cf. \Cref{prop:48076}, \Cref{coro:3935}}]\label{thm:25467}
	Let $\bfA_{\FT,K}^{\np,?}\in\left\{\bfA_{\frakF,K}^{\np},\bfA_{\frakF,K}^{\np,\rmc}\right\}$. The functor
	$$\left\{\text{\'etale }(\varphi,\Gamma)\hyphen\text{modules over }\bfA_K\right\}\lto\left\{\text{\'etale }(\varphi,\tau)\hyphen\text{modules over }\left(\bfA_{\tau,K},\bfA_{\FT,K}^{\np,?}\right)\right\},$$
	$$D\longmapsto \left(\bfA_{\FT,K}^{\np,?}\otimes_{\bfA_K}D\right)^{\scrG_{K_\infty}}$$
	induces the equivalence of categories, with the quasi-inverse given by
	$$D^\prime\longmapsto \left(\bfA_{\FT,K}^{\np,?}\otimes_{\bfA_{\tau,K}}D^\prime\right)^{\scrG_{K_{p^\infty}}}.$$
\end{introthm}

\subsection*{Herr-Ribeiro complex and Galois cohomology}
By Fontaine's equivalence of categories, the Galois cohomology of $p$-adic representation of $\scrG_K$ can be computed via $(\varphi,\Gamma)$-modules and its variants:%
\begin{enumerate}
	\item In \cite[Théorème 2.1]{herrCohomologieGaloisienneCorps1998}, Herr proved that the Galois cohomology $\rmH^i(\scrG_K,V)$ of $V\in\repcat$ can be computed by the complex $\opn{Kos}\left(\varphi,\Gamma_K,\bfD(V)\right)\coloneqq \opn{Tot}\left(\scrC_\gamma(\bfD(V))\right)$ of \'etale $(\varphi,\Gamma)$-modules $\bfD(V)$ over $\bfA_K$, where
	      $$\scrC_\gamma(\bfD(V))\coloneqq\begin{tikzcd}
			      \bfD(V)\ar[r,"1-\varphi"]\ar[d,"\gamma-1"]&\bfD(V)\ar[d,"\gamma-1"]\\\bfD(V)\ar[r,"1-\varphi"]&\bfD(V)
		      \end{tikzcd}$$
	      and $\gamma$ is a topological generator of $\Gamma_K$.
	\item In \cite[Theorem 1.1.13]{Zhao2022}, Zhao proved that $\rmH^i(\scrG_K,V)$ can be computed by the complex
	      $\opn{Kos}\left(\varphi,\wtilde{\Gamma}_{\tau,K},\wtilde{\bfD}_\tau(V)\right)\coloneqq \opn{Tot}\left(\scrC_{\tau}(\wtilde{\bfD}_\tau(V))\right)$
	      of \'etale $(\varphi,\tau)$-modules $\wtilde{\bfD}_\tau(V)$ over $\left(\bfA_{\tau,K},W\left(\wtilde{\bfE}_{\frakF,K}\right)\right)$, where
	      $$\scrC_{\tau}\left(\wtilde{\bfD}_\tau(V)\right)\coloneqq \begin{tikzcd}
			      \wtilde{\bfD}_\tau(V)\ar[r,"\tau-1"]&\left(W\left(\wtilde{\bfE}_{\frakF,K}\right)\otimes_{\bfA_{\tau,K}}\wtilde{\bfD}_\tau(V)\right)^{\gamma-\delta=0}\\
			      \wtilde{\bfD}_\tau(V)\ar[r,"\tau-1"]\ar[u,"1-\varphi"]&\left(W\left(\wtilde{\bfE}_{\frakF,K}\right)\otimes_{\bfA_{\tau,K}}\wtilde{\bfD}_\tau(V)\right)^{\gamma-\delta=0}\ar[u,"1-\varphi"]
		      \end{tikzcd}$$
	      and $\tau$ is a topological generator of $\wtilde{\Gamma}_{\tau,K}\coloneqq \Gal(K_{\frakF}/K_{p^\infty})$.
	\item In \cite[Theorem 1.5]{ribeiroExplicitFormulaHilbert2011}, Ribeiro proved that $\rmH^i(\scrG_K,V)$ can also be computed by the complex $\opn{Kos}\left(\varphi,\Gamma_{\frakF,K},\bfD_{\frakF}^{\ttR}(V)\right)\coloneqq \opn{Tot}\left(\scrC_{\frakF}\left(\bfD_{\frakF}^{\ttR}(V)\right)\right)$
	      of \'etale $(\varphi,\Gamma_{\frakF,K})$-modules over $\bfA_{\frakF,K}^{\ttR}$, where
	      $$\scrC_{\frakF}\left(\bfD_{\frakF}^{\ttR}(V)\right)\coloneqq\begin{tikzcd}[column sep=small]%
			      & \bfD_{\frakF}^{\ttR}(V)\ar[rr,"\tau-1"{xshift=0pt}]\ar[dd,"\gamma-1"{yshift=17pt}] && \bfD_{\frakF}^{\ttR}(V)\ar[dd,"\gamma-\delta"{yshift=0pt}] \\
			      \bfD_{\frakF}^{\ttR}(V)\ar[ur,"1-\varphi"]\ar[dd,"\gamma-1"{yshift=0pt}]  && \bfD_{\frakF}^{\ttR}(V) \ar[ur,"1-\varphi"]\\
			      & \bfD_{\frakF}^{\ttR}(V)\ar[rr,"\tau^{\chi(\gamma)}-1"{xshift=-17pt}] && \bfD_{\frakF}^{\ttR}(V) \\
			      \ar[rr,"\tau^{\chi(\gamma)}-1"{xshift=0pt}]\bfD_{\frakF}^{\ttR}(V)\ar[ur,"1-\varphi"]&& \bfD_{\frakF}^{\ttR}(V)\ar[ur,"1-\varphi"]
			      \arrow[from=2-3,to=4-3,"\gamma-\delta"{yshift=17pt},crossing over]
			      \arrow[from=2-1, to=2-3,"\tau-1"{xshift=-17pt},crossing over]
		      \end{tikzcd}$$
	      and $\delta=\frac{\tau^{\chi(\gamma)}-1}{\tau-1}$.
\end{enumerate}
In the spirit of \Cref{thm:25467}, we show that after replacing $W\left(\wtilde{\bfE}_{\frakF,K}\right)$ and $\bfA_{\frakF,K}^{\ttR}$ with our multivariable period rings $\bfA_{\frakF,K}^{\np}$ or $\bfA_{\frakF,K}^{\np,\rmc}$ in the aforementioned complexes, these complexes can be naturally connected without the bridging of $\rmR\Gamma_{\opn{cont}}(\scrG_K,V)$:
\begin{introthm}[{cf. \Cref{{thm:26607}}}]\label{thm:59311}
	Let $\bfD_{\frakF}^{\np,?}(V)$ be the \'etale $(\varphi,\Gamma_{\frakF,K})$-module associated to $V$ over $\bfA_{\frakF,K}^{\np,?}\in\left\{\bfA_{\frakF,K}^{\np},\bfA_{\frakF,K}^{\np,\rmc}\right\}$.
	\begin{enumerate}
		\item $\scrC_\gamma(\bfD(V))$ is the lateral kernel complex of $\scrC_{\frakF}\left(\bfD_{\FT}^{\np,?}(V)\right)$, which induces injection of complexes
		      $$\iota_{p^\infty}\colon\opn{Kos}\left(\varphi,\Gamma_K,\bfD(V)\right)\longhookrightarrow \opn{Kos}\left(\varphi,\Gamma_{\FT,K},\bfD_{\FT}^{\np,?}(V)\right)$$
		      that induces quasi-isomorphism.
		\item $\scrC_\tau\left(\bfD_\tau^{\np,?}(V)\right)$ is the vertical kernel complex of $\scrC_{\frakF}\left(\bfD_{\FT}^{\np,?}(V)\right)$, which induces injection of complexes
		      $$\iota_{\infty}\colon\opn{Kos}\left(\varphi,\wtilde{\Gamma}_{\tau,K},\bfD_{\tau}^{\np,?}(V)\right)\longhookrightarrow \opn{Kos}\left(\varphi,\Gamma_{\FT,K},\bfD_{\FT}^{\np,?}(V)\right)$$
		      that induces quasi-isomorphism.
		\item The following diagram commutes:
		      $$\begin{tikzcd}[row sep=huge]
				      \opn{Kos}\left(\varphi,\Gamma_K,\bfD(V)\right)\ar[r,"\cong"', "\iota_{p^\infty}^*"]\ar[rd,"\cong","\text{Herr}"']&\opn{Kos}\left(\varphi,\Gamma_{\FT,K},\bfD_{\FT}^{\np,?}(V)\right)\ar[d,"\cong","\text{Ribeiro}"']&\opn{Kos}\left(\varphi,\wtilde{\Gamma}_{\tau,K},\bfD_{\tau}^{\np,?}(V)\right)\ar[l,"\iota_{\infty}^*"',"\cong"]\ar[ld,"\cong","\text{Zhao}"']\\
				      &\rmR\Gamma_{\opn{cont}}(\scrG_K,V)&
			      \end{tikzcd}$$
	\end{enumerate}
\end{introthm}
\begin{remark}
	When finishing this paper independently, we realize that our approach to this theorem is partially covered by the papers \cite{ZHAO202566} and \cite{gao2024cohomologyvarphitaumodules}, which are extensions of the preprint \cite{Zhao2022}, in the very recent literature.
	\begin{enumerate}
		\item In \cite[Corollary 4.14]{ZHAO202566}, Zhao describes the $1$-cocycle of $\rmH^1(\scrG_K,V)$ corresponding to $\rmH^1\left(\opn{Kos}\left(\varphi,\wtilde{\Gamma}_{\tau,K},\wtilde{\bfD}_{\tau}(V)\right)\right)$ by showing that the morphism
		      $$\iota_{\infty}^{\opn{Zhao}}\colon\opn{Kos}\left(\varphi,\wtilde{\Gamma}_{\tau,K},\wtilde{\bfD}_{\tau}(V)\right)\lto \opn{Kos}\left(\varphi,\Gamma_{\FT,K},W\left(\wtilde{\bfE}_{\frakF,K}\right)\otimes_{\bfA_{\tau,K}}\wtilde{\bfD}_{\tau}(V)\right)$$
		      induces quasi-isomorphism. While in this article, we go in the opposite direction: showing that $\iota_\infty$ induces quasi-isomorphism by comparing the cocycles\footnote{The $1$-cocycle given by \cite[Corollary 4.14]{ZHAO202566} is an analogous result to \cite[Proposition I.4.1]{cherbonnierTheorieIwasawaRepresentations1999} and can be proved similarly.}.
		\item When proving \cite[Theorem 1.9]{gao2024cohomologyvarphitaumodules}, Gao and Zhao also make a similar observation as we make in diagram \eqref{eq:46495}: Zhao's complex can be realized by taking (the total complex of) the kernel complex of a triple complex, whose totalization is Ribeiro's complex.
	\end{enumerate}
	Despite the above overlap, we still find our result valuable in the following sense:
	\begin{enumerate}
		\item When Zhao proves that $\iota_{\infty}^{\mathtt{Zhao}}$ induces a quasi-isomorphism (cf. \cite[Theorem 4.11]{ZHAO202566}), he needs to use an intermediate complex, with our terminology, $\opn{Kos}\left(\varphi,\wtilde{\Gamma}_{\tau,K},W\left(\wtilde{\bfE}_{\tau,K}\right)\otimes_{\bfA_{\tau,K}} \wtilde{\bfD}_{\tau}(V)\right)$. This is caused by the fact that $$\left(W\left(\wtilde{\bfE}_{\frakF,K}\right)\otimes\wtilde{\bfD}_{\tau}(V)\right)^{\scrG_{K_\infty}}\cong W\left(\wtilde{\bfE}_{\tau,K}\right)\otimes_{\bfA_{\tau,K}} \wtilde{\bfD}_{\tau}(V),$$
		      and consequently an additional ``descent'' from $W\left(\wtilde{\bfE}_{\tau,K}\right)\otimes \wtilde{\bfD}_{\tau}(V)$ to $\wtilde{\bfD}_{\tau}(V)$ is needed. However, since \Cref{thm:25467} ensures that $\bfA_{\frakF,K}^{\np,?}\otimes_{\bfA_{\tau,K}}\bfD_{\tau}^{\np,?}(V)\cong \bfD_{\frakF}^{\np,?}(V)$, we do not need any intermediate complex for proving \Cref{thm:59311}.
		\item Our cocycle-comparing approach to \Cref{thm:59311} can also be adapted to \'etale $(\varphi,\Gamma_{\frakF,K})$-modules over period rings other than $\bfA_{\frakF,K}^{\np}$ and $\bfA_{\frakF,K}^{\np,\rmc}$.%

	\end{enumerate}
\end{remark}

\subsection*{$\psi$-operator and discussion on Iwasawa cohomology}
The imperfectness of multivariable period rings $\bfA_{\frakF,K}^{\np}$ and $\bfA_{\frakF,K}^{\np,\rmc}$ allows us to define the $\psi$-operator on the \'etale $(\varphi,\Gamma_{\frakF,K})$-modules:
\begin{definition}[{cf. \Cref{prop:2140}}]
	For every \'etale $(\varphi,\Gamma_{\FT,\Qp})$-module $D$ over $\bfA_{\FT,\Qp}^{\np,?}$, there exists a unique operator $\psi^{\np}\colon D\lto D$ satisfying
	$$\psi^{\np}(a\cdot\varphi(x))=a\cdot\psi^{\np}(x),\ \psi^{\np}(a\cdot\varphi(x))=\psi^{\np}(a)\cdot x$$
	for all $a\in \bfA_{\FT,\Qp}^{\np,?}$ and $x\in D$.
\end{definition}
If one wants to imitate the methodology in \cite{cherbonnierTheorieIwasawaRepresentations1999} to compute the Iwasawa cohomology of $p$-adic Galois representation $V$ via \'etale $(\varphi,\Gamma)$-modules for false Tate curve extension, one needs to prove that the complex $\opn{Kos}\left(\psi^{\np},\Gamma_{\FT,K},\bfD_{\FT}^{\np,?}(V)\right)$, which is obtained by replacing $\varphi$ by $\psi^{\np}$ in the complex $\opn{Kos}\left(\varphi,\Gamma_{\FT,K},\bfD_{\FT}^{\np,?}(V)\right)$, is quasi-isomorphic to $\opn{Kos}\left(\varphi,\Gamma_{\FT,K},\bfD_{\FT}^{\np,?}(V)\right)$. Here $\bfD_{\FT}^{\np,?}(V)$ is the $(\varphi,\Gamma_{\FT,K})$-module associated to $V$ over $\bfA_{\FT,K}^{\np,?}\in\left\{\bfA_{\frakF,K}^{\np},\bfA_{\frakF,K}^{\np,\rmc}\right\}$. We prove that:
\begin{proposition}[{\Cref{prop:16574}}]
	The natural morphism of complexes $\Phi$ in \Cref{ques:56379} induces isomorphisms of cohomology groups
	$$\rmH^0\left(\opn{Kos}\left(\varphi,\Gamma_{\FT,K},\bfD_{\FT}^{\np,?}(V)\right)\right)\cong\rmH^0\left(\opn{Kos}\left(\psi^{\np},\Gamma_{\FT,K},\bfD_{\FT}^{\np,?}(V)\right)\right).$$
\end{proposition}
On the other hand, proving $\Phi$ induces an isomorphism of $\rmH^1$ and $\rmH^2$ seems to be difficult. We discuss the obstacles and possibility at the end of \Cref{sec:6180}. In particular, we estimates the over-complicated perfectoid valuation $v^\flat$ on $\bfE_{\frakF,\Qp}^{\np,\circ,+}$ in terms of the monomial valuation
$$v_{\opn{mono}}\left(\sum_{i,j\geq 0}f_{ij}u_{\Qp}^i\eta_{\Qp}^j\right)\coloneqq \min_{i,j\geq 0,f_{ij}\neq 0}i\cdot v^\flat(u_K)+j\cdot v^\flat(\eta_K)$$
in the appendix:
\begin{introthm}[{\Cref{thm:mono}}]
	Let
	\begin{align*}
		\frakM(t)\coloneqq & \frac{p^{2\lceil t\rceil+3}}{p-1}\prod_{k=0}^{\lceil t\rceil}\left(k+\frac{p+1}{p}\right)^2-(p-1)\sum_{k=0}^{\lceil t\rceil}\left(k+1\right)p^k\prod_{l=0}^{k-1}\left(l+\frac{p+1}{p}\right) \\
		\colon             & \bfR_{\geq 0}\to\bfQ_{>0}.
	\end{align*}
	For any $f\in \bfE_{\frakF,\Qp}^{\np,\circ,+}$ with $v^\flat(f)\geq \frakM(0)$, one has
	$$v_{\opn{mono}}(f)\leq v^\flat(f)<\frakM\left(v_{\opn{mono}}(f)\right)=O\left(p^{2\cdot v_{\opn{mono}}(f)}\cdot \Gamma\left(v_{\opn{mono}}(f)+C\right)^2\right),$$
	where $C\coloneqq 3+\frac{1}{p}$ and $\Gamma(\cdot)$ is the usual Gamma function.
\end{introthm}
This might be helpful to do the explicit calculation when considering the Iwasawa cohomology.

\subsection*{Acknowledgement}
This paper is partially derived from my doctoral thesis. I would like to express my gratitude to my supervisor, Lei Fu, for his patience and support throughout this research. I deeply thank Shanwen Wang for his introducing me to this topic and for his valuable comments on the earlier version of this paper. Thank Boris Adamczewski for telling me about the existence of \cite[Example 7.2]{arocaSupportLaurentSeries2019}, which is crucial to proof \Cref{prop:29251}. Finally, thank Laurent Berger, Heng Du, Gal Porat, Léo Poyeton, Liang Xiao, Luming Zhao for the enlightening discussion and comments.%

\section{Preliminaries}
For the convenience of readers, we collect basic facts and notations on $p$-adic false Tate curve extension, $(\varphi,\Gamma)$-modules and other elements that will be used in the rest of this paper.
\subsection{$p$-adic extensions and their Galois groups}
\begin{definition}\leavevmode
	\begin{enumerate}
		\item Call $K_{p^\infty}\coloneqq \bigcup_{n=1}^\infty K\left(\zeta_{p^n}\right)$ the cyclotomic extension of $K$. Let $H_K\coloneqq \scrG_{K_{p^\infty}}$ and  $\Gamma_K\coloneqq \scrG_K/H_K$.
		\item Call $K_{\infty}\coloneqq \bigcup_{n=1}^\infty K\left(\pi_K^{1/p^n}\right)$ the Kummer extension of $K$. Let $H_{\tau,K}\coloneqq \scrG_{K_\infty}$.
		\item Call $K_{\frakF}\coloneqq \bigcup_{n=1}^\infty K\left(\pi_K^{1/p^n},\zeta_{p^n}\right)$ the false Tate curve extension of $K$. Let $H_{\frakF,K}\coloneqq \scrG_{K_\frakF}$, $\Gamma_{\frakF,K}\coloneqq \scrG_K/H_{\frakF,K}$, $\wtilde{\Gamma}_{K}\coloneqq \Gal(K_{\frakF}/K_{\infty})$ and $\wtilde{\Gamma}_{\tau,K}\coloneqq \Gal(K_{\frakF}/K_{p^\infty})$.
	\end{enumerate}
\end{definition}
We summarize the above definition in the following diagram:
\[\begin{tikzcd}
		& \overline{\bfQ}_p \\
		& K_{\frakF} \\
		K_{\infty} && K_{p^\infty} \\
		& K
		\arrow[no head, from=1-2, to=2-2,"H_{\frakF,K}"]
		\arrow[curve={height=18pt}, no head, from=1-2, to=3-1, "H_{\tau,K}"']
		\arrow[curve={height=-18pt}, no head, from=1-2, to=3-3,"H_K"]
		\arrow[no head, from=2-2, to=3-1,"\wtilde{\Gamma}_K"']
		\arrow[no head, from=2-2, to=3-3,"\wtilde{\Gamma}_{\tau,K}"]
		\arrow[no head, from=2-2, to=4-2,"\Gamma_{\frakF,K}"]
		\arrow[dotted, no head, from=3-1, to=4-2]
		\arrow[no head, from=3-3, to=4-2,"\Gamma_K"]
	\end{tikzcd}\]
\begin{remark}
	The extension $K_\infty/K$ is not Galois, so there is no ``$\Gal(K_\infty/K)$''.
\end{remark}
\begin{lemma}[{cf. \cite[Lemma 5.12]{liuLatticesSemistableRepresentations2008}}]\label{lem:22831}
	One has
	\begin{enumerate}
		\item $K_\infty\cap K_{p^\infty}=K$;
		\item The Galois group $\wtilde{\Gamma}_K\cong \Gamma_K$ is identified as a subgroup of $\bfZ_p^\times$ via the cyclotomic character $\chi$, and $\wtilde{\Gamma}_{\tau,K}\cong\bfZ_p$;
		\item $\Gamma_{\frakF,K}\cong\wtilde{\Gamma}_{\tau,K}\rtimes\wtilde{\Gamma}_K$, and $\Gamma_K$ acts on $\bfZ_p$ by the cyclotomic character.
	\end{enumerate}
\end{lemma}
\begin{remark}
	According to \cite[Remark 5.1.3]{liuLatticesSemistableRepresentations2008}, \Cref{lem:22831} fails for $p=2$, and this is the reason we assume $p\geq 3$ in this paper.
\end{remark}
\subsection{Field of norms and action of Galois groups}
\begin{definition}[{cf. \cite[Section 2.1.2]{Wintenberger1983}}]\label{def:69346}
	For any sAPF extension (cf. \cite[Définitions 1.2.1]{Wintenberger1983}) $L/K$, define the field of norms
	$$X_K(L)\coloneqq \{0\}\cup\varprojlim_{M\in\scrE(L/K)}M^\times,$$
	where $\scrE(L/K)$ is the set of finite extensions of $K$ contained in $L$ and the transition map is the norm map.
\end{definition}
\begin{remark}
	It can be proved that all the extensions $K_\infty$, $K_{p^\infty}$ and $K_{\frakF}$ are sAPF extensions of $K$. See \cite[Proposition 1.2.3]{Wintenberger1983} and \cite[Main Theorem]{senRamificationPadicLie1972} for more details.
\end{remark}
\begin{theorem}[{cf. \cite[Théorème 2.1.3, Proposition 4.2.1]{Wintenberger1983}}]\label{thm:55903}
	Let $L/K$ be a sAPF extension.
	\begin{enumerate}
		\item With addition defined by
		      \begin{equation*}\left(\alpha_M\right)_{M\in \scrE(L/K)}+\left(\beta_M\right)_{M\in\scrE(L/K)}\coloneqq \left(\lim_{N\in\scrE(L/M)}\calN_{N/M}(\alpha_N+\beta_N)\right)_{M\in\scrE(L/K)},\end{equation*}
		      and component-wise multiplication, $X_K(L)$ is a complete discrete valued field of characteristic $p$ with residue field $k_L$ and valuation defined by $$v\left(\left(\alpha_M\right)_{M\in \scrE(L/K)}\right)\coloneqq v_p(\alpha_{K_0}),$$
		      where $K_0\in\scrE(L/K)$ is the maximal unramified extension of $K$ contained in $L$. Moreover, the absolute Galois group is naturally isomorphic to $\scrG_L$.
		\item The field of norms $X_K(L)$ embeds continuously into the perfectoid field $\widehat{L}^\flat\coloneqq \varprojlim_{x\mapsto x^p}L/p$. More specifically, Let $K_1$ be the maximally tamely ramified extension of $K$ in $L$. For any positive integer $n$, let
		      $$\scrE_n=\{M\in\scrE(L/K_1)\colon v_p([M\colon K_1])\geq n\}.$$
		      For any element $(\alpha_E)_{E\in\scrE(L/K)}$ in $X_K(L)$,
		      $$\left\{\alpha_E^{[E\colon K_1]\cdot p^{-n}}\right\}_{E\in\scrE_n}$$
		      converges to an element $x_n$ in $\widehat{L}$ with respect to the filter $(\scrE_\bullet,\subseteq)$, and the map
		      $$\Lambda_{L/K}\colon X_K(L)\lto \widehat{L}^\flat,\ (\alpha_E)_{E\in\scrE(L/K)}\longmapsto (x_n)_{n\geq 1}$$ is continuous and injective.
	\end{enumerate}
\end{theorem}
\begin{definition}\leavevmode
	\begin{enumerate}
		\item Let $\bfE_K=\opn{Im}\Lambda_{K_{p^\infty}/K}$, $\bfE_{\tau,K}\coloneqq \opn{Im}\Lambda_{K_\infty/K}$ and $\bfE_{\frakF,K}\coloneqq \opn{Im}\Lambda_{K_\frakF,K}$.%
		      Let $\wtilde{\bfE}_K\coloneqq \widehat{K}_{p^\infty}^\flat$, $\wtilde{\bfE}_{\tau,K}\coloneqq \widehat{K}_{\infty}^\flat$ and $\wtilde{\bfE}_{\frakF,K}\coloneqq \widehat{K}_{\frakF}^\flat$.
		\item Fix the uniformizer of $\bfE_{\tau,K}$ to be $u_K\coloneqq \left(u_K,u_K^{1/p},\cdots,u_K^{1/p^n},\cdots\right)\in\Cpflat$ and denote by $\eta_K$ the uniformizer of $\bfE_K$. When $K=\Qp$, we take $\eta_{\Qp}=\varepsilon-1$, where $\varepsilon=\left(1,\zeta_p,\cdots,\zeta_{p^n},\cdots\right)\in\Cpflat$. Since $\bfE_{\tau,K}$ is a subfield of $\wtilde{\bfE}_{\frakF,K}$, elements of $\wtilde{\Gamma}_{\tau,K}$ acts on $\bfE_{\tau,K}$ by sending $u_K$ to $u_K\cdot \varepsilon^a$ for some $a\in\bfZ_p^\times$. We fix a topological generator $\tau_K$ of $\wtilde{\Gamma}_{\tau,K}$ such that $\tau_K\cdot u_K=u_K\cdot \varepsilon$.
		\item Let $\bfA_K$ be the $p$-Cohen ring of $\bfE_K$ in $\bfA_{\opn{inf}}$ that is equipped with a continuous action of $\Gamma_K$ that lifts the action of it on $\bfE_K$ (see for instance \cite[Theorem 1.3]{bergerLiftingFieldNorms2014}). Let $\bfA$ be the completion of the maximal unramified ring extension\footnote{We say $R_1$ is the maximal unramified ring extension of $R_2$ if $R_1$ is the ring of integers of the maximal unramified extension of $\opn{Fr}R_1$ in $\opn{Fr}\bfA_{\opn{inf}}$.} of $\bfA_K$. Let $\bfA_{\tau,K}$ be the $p$-adic completion of the ring $W(k_K)\bbrac{[u_K]}[1/[u_K]]$ and let $\bfA_{\frakF,K}^{\ttR}\coloneqq \bfA^{H_{\frakF,K}}$.
	\end{enumerate}
\end{definition}
\begin{lemma}\label{lem:13864}\leavevmode
	\begin{enumerate}
		\item There exists a topological generator $\gamma_K$ of $\Gamma_K$ and a formal power series $\frakG_K(T)\in k_K\bbrac{T}^\times$ such that the action of $\Gamma_K$ on $\bfE_K$ is given by the formula
		      $$\gamma_K\colon \eta_K\longmapsto \eta_K\cdot \frakG_K(\eta_K).$$
		      When $K=\bfQ_p$, we can take $\frakG_K(T)=\frac{1}{T}\left((1+T)^{\chi\left(\gamma_{\bfQ_p}\right)}-1\right)$.
		\item There exists formal power series $\frakT_K(T)\in k_K\bbrac{T}$ such that
		      $$\tau_K\cdot u_K=u_K\cdot \left(1+\eta_K\cdot \frakT_K(\eta_K)\right).$$
		      When $K=\Qp$, then $\frakT_K(T)=1$.
	\end{enumerate}
\end{lemma}
\begin{proof}
	\begin{enumerate}
		\item The Galois group $\Gamma_K$ acts isometrically on $\bfE_K$, thus $v^\flat\left(\frac{\gamma_K\cdot\eta_K}{\eta_K}\right)=1$. The existence of $\frakG_K(T)$ follows from the fact that $\bfE_K$ is a local field. When $K=\Qp$, the formula is well-known.
		\item By the inclusion $\bfE_{\Qp}\subseteq \opn{Im}\Lambda_{K_{p^\infty}/K}\subseteq \bfE_K$, $\eta_{\Qp}$ is an element of $\bfE_K$ with positive valuation. Thus, there exists a formal power series $\frakT_K(T)\in k_K\bbrac{T}$ such that $\varepsilon-1=\eta_K\cdot \frakT_K(\eta_K)$. The result follows from our choice of $\tau_K$.
	\end{enumerate}
\end{proof}
\begin{remark}
	We abbreviate the element $\gamma_K$ (resp. $\tau_K$) to $\gamma$ (resp. $\tau$) if the choice of $K$ is self-explanatory in the context.
\end{remark}
\subsection{Variants of \'etale $(\varphi,\Gamma)$-modules}
\begin{definition}\label{def:51763}
	Let $$(\bfA_1,\bfA_2)\in\left\{\left(\bfA_K,\bfA_K\right),\left(\bfA_{\tau,K},W\left(\wtilde{\bfE}_{\frakF,K}\right)\right),\left(\bfA_{\frakF,K}^{\ttR},\bfA_{\frakF,K}^{\ttR}\right)\right\}.$$
	\begin{enumerate}
		\item A $\varphi$-module over $\bfA_1$ is a finitely generated $\bfA_1$-module endowed with a semilinear action of the Frobenius map $\varphi$. A $\varphi$-module $D$ is called \'etale if $\varphi^*(D)\cong D$.
		\item Let $\boldsymbol{\Phi\Gamma}^{\text{\'et}}(\bfA_1,\bfA_2)$ be the category with the objects given by \'etale $\varphi$-modules $D$ over $\bfA_1$ that satisfy:
		      \begin{enumerate}
			      \item $\scrG_{\opn{Fr \bfA_2}}$ acts continuously and semilinearly on $\bfA_2\otimes_{\bfA_1}D$.
			      \item Regarding $D$ as an $\bfA_1$-submodule of $\bfA_2\otimes_{\bfA_1}D$, then $D\subseteq \left(\bfA_2\otimes_{\bfA_1}D\right)^{\scrG_{\opn{Fr}\bfA_1}}$.
		      \end{enumerate}
	\end{enumerate}
\end{definition}
\begin{example}\label{eg:5636}\leavevmode
	\begin{enumerate}
		\item Call objects in $\boldsymbol{\Phi\Gamma}^{\text{\'et}}\left(\bfA_K,\bfA_K\right)$ the (\'etale) $(\varphi,\Gamma_K)$-modules over $\bfA_K$. Here $\scrG_{\opn{Fr}\bfA_1}\cong H_K$.
		\item Call objectives in $\boldsymbol{\Phi\Gamma}^{\text{\'et}}\left(\bfA_{\tau,K},W\left(\wtilde{\bfE}_{\frakF,K}\right)\right)$ the (\'etale) $(\varphi,\tau)$-modules over $\left(\bfA_{\tau,K},W\left(\wtilde{\bfE}_{\frakF,K}\right)\right)$. Here $\scrG_{\opn{Fr}\bfA_2}\cong H_{\frakF,K}$ and $\scrG_{\opn{Fr}\bfA_1}\cong H_{\tau,K}$.
		\item Call objects in $\boldsymbol{\Phi\Gamma}^{\text{\'et}}\left(\bfA_{\frakF,K}^{\ttR},\bfA_{\frakF,K}^{\ttR}\right)$ the \'etale $(\varphi,\Gamma_{\frakF,K})$-modules over $\bfA_{\frakF,K}^{\ttR}$. Here $\scrG_{\opn{Fr}\bfA_{\frakF,K}^{\ttR}}\cong H_{\frakF,K}$.
	\end{enumerate}
\end{example}
\begin{remark}
	As we will see later, the terminologies we introduced in \Cref{def:51763} and \Cref{eg:5636} extend to period rings $(\bfA_1,\bfA_2)$ that are not in the list of \Cref{def:51763}. To keep the notations simple, these terminologies will still be used in the new context.
\end{remark}
\begin{theorem}[{\cite[Théorème A.3.4.3]{Fontaine1990}, \cite[Section 1.3]{Caruso2013}, \cite[Theorem 1.3]{ribeiroExplicitFormulaHilbert2011}}]\label{thm:36449}
	For any $(\bfA_1,\bfA_2)$ in \Cref{def:51763}, let $\bfA^{\opn{ur}}$ be the completion of maximal unramified ring extension of $\bfA_2$ in $\bfA_{\opn{inf}}$. The functor
	$$\bfD_{\bfA_1,\bfA_2}\colon \repcat\lto \boldsymbol{\Phi\Gamma}^{\text{\'et}}(\bfA_1,\bfA_2),\ V\longmapsto \left(\bfA\otimes_{\bfZ_p}V\right)^{\scrG_{\opn{Fr}\bfA_1}}$$
	induces an equivalence of categories,
	with a quasi-inverse given by $D\longmapsto \bfV(D)\coloneqq \left(\bfA^{\opn{ur}}\otimes_{\bfA_2}D\right)^{\varphi=1}$.
\end{theorem}

\begin{remark}
	We abbreviate the functor $\bfD_{\bfA_1,\bfA_2}(V)$ in \Cref{thm:36449} as:
	\begin{enumerate}
		\item $\bfD(V)$ for \'etale $(\varphi,\Gamma)$-modules;
		\item $\wtilde{\bfD}_{\tau}(V)$ for \'etale $(\varphi,\tau)$-modules over $\left(\bfA_{\tau,K},W\left(\wtilde{\bfE}_{\frakF,K}\right)\right)$;
		\item $\bfD_{\frakF}^{\ttR}(V)$ for \'etale $(\varphi,\Gamma_{\frakF,K})$-modules over $\bfA_{\frakF,K}^{\ttR}$.
	\end{enumerate}
\end{remark}

\section{Generalized Caruso's embedding lemma}
For any elements $\alpha,\beta\in\scrO_{\bfC_p}^\flat$ and any $f(X,Y)=\sum_{i,j\geq 0}f_{ij}X^iY^j\in k_K\bbrac{X,Y}$, let
$$f(\alpha,\beta)\coloneqq \lim_{n\to\infty} \sum_{i=0}^n\sum_{j=0}^n f_{ij}\alpha^i\beta^j\in\scrO_{\bfC_p}^\flat.$$
Recall the following result of Caruso:
\begin{lemma}[Caruso's embedding lemma, {\cite[Proposition 1.7]{Caruso2013}}]
	The map
	$$\iota_{\mathtt{Car}}\colon k_K\bbrac{X,Y}\lto \scrO_{\bfC_p}^\flat,\ f\longmapsto f\left(u_K,\eta_\Qp\right)$$
	is injective.
\end{lemma}
The main goal of this section is the following (variant of) generalization of Caruso's embedding lemma. This new version will specialize to the original one when $K=\Qp$, and will play an important role in studying the topological properties of the multivariable period rings in the next section.
\begin{proposition}\label{prop:60072}\label{lem:62550}
	For any $\theta\in\bfZ_{\geq 1}$ and $F(X)\in k_K\bbrac{X}$ satisfying $v^\flat\left(\frac{\eta_K^\theta}{F(u_K)}\right)>0$, define the map
	$$\iota_{\theta,F}\colon k_K\bbrac{X,Y}\lto \scrO_{\bfC_p}^\flat,\ f\longmapsto f\left(u_K,\frac{\eta_K^\theta}{F(u_K)}\right).$$
	For any $n\in\bfN$, there exists $M(n)\in\bfQ_{>0}$ such that for every formal power series $f(X,Y)=\sum_{i,j\geq 0}f_{ij}X^i Y^j\in k_K\bbrac{X,Y}$ satisfying $v^\flat\left(\iota_{\theta,F}(f)\right)>M(n)$, one has $f_{ij}=0$ for every $(i,j)\in \bfN_{\leq n}^2$. In particular, the map $\iota_{\theta,F}$ is injective.
\end{proposition}
The main ingredient of the proof of \Cref{prop:60072} is the following lemma:
\begin{lemma}\label{prop:61330}
	Using the notations in \Cref{prop:60072}. For any $s\in \bfN$, denote by $\partial_1^{[s]}\coloneqq\frac{1}{s!}\frac{\partial^s}{\partial X^s}$ (resp. $\partial_2^{[s]}\coloneqq\frac{1}{s!}\frac{\partial^s}{\partial Y^s}$) the Hasse derivative on $k_K\bbrac{X,Y}$ with respect to $X$ (resp. $Y$).
	Then
	\begin{enumerate}
		\item For any $i\in\bfN$ and $c\in\bfQ_{>0}$, there exists $M_1^{[i]}(c)\in\bfQ_{>0}$ such that for any formal power series $f\in k_K\bbrac{X,Y}$ with $v^\flat\left(\iota_{\theta,F}(f)\right)\geq M_1^{[i]}(c)$, one has $v^\flat\left(\iota_{\theta,F}\left(\partial_1^{[s]}f\right)\right)\geq c$ for $s=0,\cdots,i$.
		\item For any $j\in\bfN$ and $c\in\bfQ_{>0}$, there exists $M_2^{[j]}(c)\in\bfQ_{>0}$ such that for any formal power series $f\in k_K\bbrac{X,Y}$ with $v^\flat\left(\iota_{\theta,F}(f)\right)\geq M_2^{[j]}(c)$, one has $v^\flat\left(\iota_{\theta,F}\left(\partial_2^{[s]}f\right)\right)\geq c$ for $s=0,\cdots,j$.
	\end{enumerate}
\end{lemma}
\begin{proof}
	We first prove the second statement. When $j=0$ we can just take $M_2^{[0]}(c)=c$. Suppose we have constructed such $M_2^{[j]}(c)$ for some $j\geq 0$.

	Since $\Gamma_{\FT,K}$ acts continuously and isometrically on $\EtmKnpc\subset\wtilde{\bfE}_{\FT,K}$ (cf. \cite[48]{schneiderGaloisRepresentationsPhi2017}), there exists $\gamma_c\in\Gamma_K\subsetneq \Gamma_{\FT,K}$ that is close enough but not equal to $1$ such that the element
	$$\delta_c\coloneqq \frac{\gamma_c\cdot \eta_K^\theta-\eta_K^\theta}{F(u_K)}$$
	has valuation not less than $c$. Define
	\begin{equation}\label{eq:27570}
		M_2^{[j+1]}(c)\coloneqq \max\left\{M_2^{[j]}\left(c+(j+1)v^\flat\left(\delta_c\right)\right),c+(j+1)v^\flat\left(\delta_c\right),M_2^{[j]}(c)\right\}
	\end{equation}
	and take $f\in k_K\bbrac{X,Y}$ satisfying $v^\flat\left(\iota_{\theta,F}(f)\right)\geq M_2^{[j+1]}(c)$.

	By Taylor expansion, we have
	\begin{equation}\label{eq:13475}
		\begin{aligned}
			\delta_c^{-(j+1)}\left(\gamma_c\cdot\iota_{\theta,F}(f)-\sum_{l=0}^j\iota_{\theta,F}\left(\partial_2^{[l]}f\right)\delta_c^l\right)= & \delta_c^{-(j+1)}\left(f\left(u,\frac{\eta_K^\theta}{F(u_K)}+\delta_c\right)-\sum_{l=0}^j\iota_{\theta,F}\left(\partial_2^{[l]}f\right)\delta_c^l\right) \\
			=                                                                                                                                    & \iota_{\theta,F}\left(\partial_2^{[j+1]}f\right)+\delta_c\sum_{l=j+2}^\infty\iota_{\theta,F}\left(\partial_2^{[l]}f\right)\delta_c^{l-(j+2)}.
		\end{aligned}
	\end{equation}
	On one hand, we have
	$$v^\flat\left(\gamma_c\cdot\iota_{\theta,F}(f)\right)=v^\flat\left(\iota_{\theta,F}(f)\right)\geq M_2^{[j+1]}(c)\geq c+(j+1)v^\flat(\delta_c).$$
	On the other hand, the condition
	$$v^\flat(\iota_{\theta,F}(f))\geq M_2^{[j]}\left(c+(j+1)v^\flat(\delta_c)\right)$$
	implies that $v^\flat\left(\iota_{\theta,F}\left(\partial_2^{[l]}f\right)\right)\geq c+(j+1)v^\flat(\delta_c)$ for $l=0,\cdots,j$. Thus, the valuation of the left-hand side of \Cref{eq:13475} is not less than $c$. By the ultra-metric inequality, we have
	\begin{align*}
		     & v^\flat\left(\iota_{\theta,F}\left(\partial_2^{[j+1]}f\right)\right)                                                                                                                                                                                                                     \\
		\geq & \min\left(v^\flat\left(\delta_c^{-(j+1)}\left(\gamma_c\cdot\iota_{\theta,F}(f)-\sum_{l=0}^j\iota_{\theta,F}\left(\partial_2^{[l]}f\right)\delta_c^l\right)\right),v^\flat\left(\delta_c\sum_{l=j+2}^\infty\iota_{\theta,F}\left(\partial_2^{[l]}f\right)\delta_c^{l-(j+2)}\right)\right) \\
		\geq & \min\left(c,v^\flat(\delta_c)\right)                                                                                                                                                                                                                                                     \\
		=    & c.
	\end{align*}
	Finally, the condition $M_2^{[j+1]}(c)\geq M_2^{[j]}(c)$ implies that $v^\flat\left(\iota_{\theta,F}\left(\partial_2^{[l]}f\right)\right)\geq c$ for $l=0,\cdots,j$. This completes the proof of the second statement.

	For the first statement, when $i=0$ we can just take $M_1^{[0]}(c)=c$. Suppose we have constructed such $M_1^{[i]}(c)$ for some $i\geq 0$. Similar to what we have done, for any $c>0$, we can take $\tau_c\in\Gal(K_{\FT}/K_{p^\infty})$ being close enough but not equal to $1$ that $v^\flat(A_c),v^\flat(B_c)\geq c$, where $A_c\coloneqq \tau_c\cdot u_K-u_K$ and
	$$B_c\coloneqq \tau_c\cdot \frac{\eta_K^\theta}{F(u_K)}-\frac{\eta_K^\theta}{F(u_K)}=-\frac{\eta_K^\theta}{F(u_K)}\cdot \frac{\tau_c\cdot F(u_K)-F(u_K)}{\tau_c\cdot F(u_K)}.$$
	Define
	\begin{equation}\label{eq:27571}
		M_1^{[i+1]}(c)\coloneqq\max\left\{c+(i+1)v^\flat(A_c),M_1^{[i]}\left(M_2^{[\lambda(c)]}\left(c+(i+1)v^\flat(A_c)\right)\right),M_1^{[i]}(c)\right\},
	\end{equation}
	where $\lambda(c)\coloneqq \left\lceil\frac{c+(i+1)v^\flat(A_c)}{v^\flat(B_c)}\right\rceil$, and take $f\in k_K\bbrac{X,Y}$ satisfying $v^\flat\left(\iota_{\theta,F}(f)\right)\geq M_1^{[i+1]}(c)$.

	One can rewrite the Taylor expansion
	$$\tau_c\cdot \iota_{\theta,F}(f)=f\left(u_K+A_c,\frac{\eta_K^\theta}{F(u)}+B_c\right)=\sum_{m=0}^\infty\sum_{l=0}^\infty \iota_{\theta,F}\left(\partial_2^{[l]}\partial_1^{[m]}f\right)B_c^lA_c^m$$
	as
	\begin{equation}\label{eq:11148}
		\iota_{\theta,F}\left(\partial_1^{[i+1]}f\right)=A_c^{-(i+1)}\left(\tau_c\cdot \iota_{\theta,F}(f)-S_1-S_2\right)-B_cS_3-A_cS_4,
	\end{equation}
	where
	$$S_1\coloneqq \sum_{m=0}^i\sum_{l=0}^{\lambda(c)}\iota_{\theta,F}\left(\partial_2^{[l]}\partial_1^{[m]}f\right)B_c^lA_c^m,\ S_2\coloneqq \sum_{m=0}^i\sum_{l=\lambda(c)+1}^\infty\iota_{\theta,F}\left(\partial_2^{[l]}\partial_1^{[m]}f\right)B_c^lA_c^m,$$
	$$S_3\coloneqq\sum_{l=1}^\infty \iota_{\theta,F}\left(\partial_2^{[l]}\partial_1^{[i+1]}f\right)B_c^{l-1}\text{ and }S_4\coloneqq \sum_{l=0}^{\infty}\sum_{m=i+2}^\infty\iota_{\theta,F}\left(\partial_2^{[l]}\partial_1^{[m]}f\right)B_c^lA_c^{m-(i+2)}.$$

	Since $$v^\flat(\iota_{\theta,F}(f))\geq M_1^{[i+1]}(c)\geq M_1^{[i]}\left(M_2^{[\lambda(c)]}\left(c+(i+1)v^\flat(A_c)\right)\right),$$
	we know that $v^\flat\left(\iota_{\theta,F}\left(\partial_1^{[m]}f\right)\right)\geq M_2^{[\lambda(c)]}\left(c+(i+1)v^\flat(A_c)\right)$ for all $m=0,1,\cdots,i$. This implies that $v^\flat\left(\iota_{\theta,F}\left(\partial_2^{[l]}\partial_1^{[m]}f\right)\right)\geq c+(i+1)v^\flat(A_c)$ for all $m=0,1,\cdots,i$ and $l=0,1,\cdots,\lambda(c)$. As a result, we obtain
	\begin{equation*}
		v^\flat\left(S_1\right)\geq\min_{\substack{0\leq m\leq i\\0\leq l\leq \lambda(c)}}v^\flat\left(\iota_{\theta,F}\left(\partial_2^{[l]}\partial_1^{[m]}f\right)\right)\geq c+(i+1)v^\flat(A_c).
	\end{equation*}
	On the other hand, the definition of $\lambda(c)$ implies that
	\begin{equation*}
		v^\flat(S_2)\geq (\lambda(c)+1)v^\flat(B_c)\geq c+(i+1)v^\flat(A_c).
	\end{equation*}
	Since
	\begin{equation*}
		v^\flat\left(\tau_c\cdot\iota_{\theta,F}(f)\right)=v^\flat\left(\iota_{\theta,F}(f)\right)\geq M_1^{[i+1]}(c)\geq c+(i+1)v^\flat(A_c)
	\end{equation*}
	and $S_3,S_4$ have non-negative valuation, we can apply the ultra-metric inequality on \Cref{eq:11148}%
	to conclude that $v^\flat\left(\iota_{\theta,F}\left(\partial_1^{[i+1]}f\right)\right)\geq c$. Finally, the condition $M_1^{[i+1]}(c)\geq M_1^{[i]}(c)$ implies that $$v^\flat\left(\iota_{\theta,F}\left(\partial_1^{[m]}f\right)\right)\geq c$$ for all $m=0,1,\cdots,i$. This completes the proof of the first statement.
\end{proof}

\begin{proof}[Proof of \Cref{prop:60072}]
	Let $M(n)\coloneqq M_1^{[n]}\left(M_2^{[n]}(1)\right)$. Then \Cref{prop:61330} implies that $$v^\flat\left(\iota_{\theta,F}\left(\partial_1^{[i]}\partial_2^{[j]}f\right)\right)\geq 1$$ for all $(i,j)\in\bfN_{\leq n}^2$. Since
	$$\partial_1^{[i]}\partial_2^{[j]}f=f_{ij}+\text{ terms with positive valuation},$$
	the ultra-metric inequality implies that $f_{ij}=0$ for all $(i,j)\in\bfN_{\leq n}^2$.

	To see the injectivity of $\iota_{\theta,F}$, we only need to take $f\in k_K\bbrac{X,Y}$ such that $\iota_{\theta,F}(f)=0$ and notice that $v^\flat(f)>M(n)$ for arbitrary $n\in\bfN$.
\end{proof}
\Cref{prop:60072} specializes to the following two important cases:
\begin{corollary}\label{coro:48511}\label{lem:1122}\label{lem:869}\leavevmode%
	\begin{enumerate}
		\item The map
		      $$\iota_{\np}\colon k_K\bbrac{X,Y}\lto \scrO_{\bfC_p}^\flat,\ f\longmapsto f\left(u_K,\eta_K\right)$$
		      is injective.
		\item Let $\theta\in\bfZ_{\geq 1}$ be a positive integer such that $v^\flat\left(\frac{\eta_K^\theta}{u_K}\right)>0$. Then the map
		      $$\iota_{\theta,X}\colon k_K\bbrac{X,Y}\lto \scrO_{\bfC_p}^\flat,\ f\longmapsto f\left(u_K,\frac{\eta_K^\theta}{u_K}\right)$$
		      is injective.
	\end{enumerate}
\end{corollary}

\section{Basic properties of $\bfE_{\FT,K}^{\np}$ and $\widehat{\bfE}_{\FT,K}^{\circ}$}
\begin{definition}\label{def:9888}
	Let $\bfE_{\FT,K}^{\np ,\circ,+}$ be the image of the embedding $\iota_{\np}$ in \Cref{lem:1122}. Let $\bfE_{\FT,K}^{\np ,\circ}$ be the fraction field of $\bfE_{\FT,K}^{\np ,\circ,+}$ and let $\widehat{\bfE}_{\FT,K}^{\np ,\circ}$ be completion of $\bfE_{\FT,K}^{\np ,\circ}$ with respect to the valuation $v^\flat$. Let $\Esep$ (resp. $\Ecsep$) be the separable closure of $\bfE_{\FT,K}^{\np ,\circ}$ (resp. $\widehat{\bfE}_{\FT,K}^{\np ,\circ}$) in $\bfC_p^\flat$.
\end{definition}

\begin{lemma}\label{lem:5466}
	The action of $H_{\frakF,K}$ on $\Esep$ (resp. $\Ecsep$) induced by the action of $\scrG_K$ on $\bfC_p^\flat$ is faithful and is trivial on $\bfE_{\FT,K}^{\np ,\circ}$ and $\widehat{\bfE}_{\FT,K}^{\np ,\circ}$, i.e., one has injective morphisms
	$$\rho\colon H_{\frakF,K}\lto \Gal\left(\Esep/\EtmKnpc\right)\text{ and\ } \widehat{\rho}\colon H_{\frakF,K}\lto \Gal\left(\Ecsep/\EtmKcnp\right).$$
\end{lemma}
\begin{proof}
	See \cite[Proposition 3.1.4, Injectivity of $\rho$]{Zhao2022}.
\end{proof}

\begin{definition}\label{def:8032}
	Let $\bfE_{\FT,K}^{\np }\coloneqq \left(\Esep\right)^{H_{\frakF,K}}$ and $\widehat{\bfE}_{\FT,K}^{\np }\coloneqq \left(\Ecsep\right)^{H_{\frakF,K}}$, where $H_{\frakF,K}$ acts on $\Esep$ (resp. $\Ecsep$) by $\rho$ (resp. $\widehat{\rho}$).
\end{definition}
\begin{remark}\label{rmk:84309}
	To simplify our notation, we denote by $\bfE_{\frakF,K}^{\np,?}$ any choice between $\EtmKnp,\EtmKnpc$ and set $\bfE_{\frakF}^{\np,?}$ to be the separable closure of $\bfE_{\frakF,K}^{\np,?}$.
\end{remark}

\subsection{Valuation-theoretic characterization of $\bfE_{\FT,K}^{\np,?}$}
In \cite[Lemma 3.1.2, Proposition 3.1.4, Corollary 3.1.5]{Zhao2022}, Zhao shows that if $\bfE_{\FT,\Qp}^{\np ,\circ}$ is complete with respect to $v^\flat$, then $\bfE_{\FT,\Qp}^{\mathtt{Car},u\hyphen\np }$ provides the ``correct'' absolute Galois group, i.e., $\scrG_{\bfE_{\FT,\Qp}^{\mathtt{Car},u\hyphen\np }}\cong H_{\frakF,\Qp}$. The same strategy can be applied to show that if $\bfE_{\FT,K}^{\np ,\circ}$ is complete with respect to $v^\flat$, then $\rho$ is an isomorphism, i.e., $\EtmKnpc=\EtmKnp$. However, the following result shows that this premise does not hold:
\begin{proposition}\label{prop:51851}
	The field $\bfE_{\FT,K}^{\np ,\circ}$ is not complete with respect to $v^\flat$, i.e., it is a proper subfield of $\widehat{\bfE}_{\FT,K}^{\np ,\circ}$.
\end{proposition}
\begin{proof}
	Let $\theta>0$ be an integer satisfying $v^\flat\left(\frac{\eta_K^\theta}{u_K}\right)>0$. Let
	$$A\coloneqq \lim_{m\to\infty}\sum_{n=0}^m C_n\cdot \left(\frac{\eta_K^\theta}{u_K}\right)^n\in \EtmKcnp,$$
	where $C_n\coloneqq \binom{2n}{n}-\binom{2n}{n+1}$ is the $n$-th Catalan number (cf. \cite[106]{koshyCatalanNumbersApplications2008}).

	One can calculate that
	$$A^2=\lim_{m\to\infty}\left(\sum_{n=0}^m C_n \left(\frac{u_K^\theta}{\eta_K}\right)^n\right)^2=\lim_{m\to\infty}\sum_{n=0}^m\left(\sum_{i=0}^n C_iC_{n-i}\right)\left(\frac{u_K^\theta}{\eta_K}\right)^n.$$
	By Segner's recurrence relation (cf. \cite[(5.6)]{koshyCatalanNumbersApplications2008}) $C_{n+1}=\sum_{i=0}^n C_iC_{n-i}$, we have
	\begin{equation}\label{eq:60689}
		A^2=\lim_{m\to\infty}\sum_{n=0}^m C_{n+1}\left(\frac{u_K^\theta}{\eta_K}\right)^n=\left(\frac{u_K^\theta}{\eta_K}\right)^{-1}\cdot (A-1).
	\end{equation}
	If $A$ lies in $\bfE_{\FT,K}^{\np ,\circ}$, then we may set $A=\iota_{np}(f)/\iota_{np}(g)$, where $f(X,Y), g(X,Y)\in k_K\bbrac{X,Y}$. \Cref{eq:60689} can be rewritten as
	$$\iota_{np}(X^\theta)\cdot \iota_{np}(f)^2-\iota_{np}(Y)\left(\iota_{np}(f)\iota_{np}(g)-\iota_{np}(g)^2\right)=0.$$
	By \Cref{lem:1122}, this implies that
	$$X^\theta\cdot f^2-Y\cdot f\cdot g+Y\cdot g^2=0,$$
	i.e.,
	$$\left(2\cdot\frac{X^\theta}{Y}\cdot\frac{f}{g}-1\right)^2=1-4\cdot\frac{X^\theta}{Y}.$$
	The contradiction follows from comparing the $Y$-adic valuation of the both sides of the above equality: the $Y$-adic valuation of the left-hand side is an even integer while the one of the right-hand side\footnote{Recall that we assume $p>2$.} is $-1$.
\end{proof}
\begin{remark}
	\Cref{prop:51851} will be covered by \Cref{prop:56478} combined with \Cref{lem:47550}, or independently by the proof of \Cref{prop:29251}. We choose to keep this proof, for \Cref{prop:51851} is the starting point of the whole study, and the proof presented here is short and elegant.
\end{remark}
Different from the single-variable case, $\bfE_{\FT,K}^{\np ,\circ,+}$ is not the ring of integers of $\bfE_{\FT,K}^{\np ,\circ}$, for elements like $\eta_K^\theta/u_K$ with $\theta>\!\!\!>0$ have positive valuation, but they do not live in $\bfE_{\FT,K}^{\np ,\circ,+}$. Therefore, \Cref{prop:51851} does not imply that $\bfE_{\FT,K}^{\np ,\circ,+}$ is incomplete with respect to $v^\flat$. On the contrary, we have:
\begin{lemma}\label{lem:2016}
	The ring $\bfE_{\FT,K}^{\np ,\circ,+}$ is complete with respect to $v^\flat$.
\end{lemma}
\begin{proof}
	Let $\left\{f_n(X,Y)=\sum_{i,j\geq 0}a_{ij}^{(n)}X^iY^j\right\}_{n\geq 1}\subset k_K\bbrac{X,Y}$ such that $\left\{\iota_{\np}(f_n)\right\}_{n\geq 1}$ is a Cauchy sequence in $\bfE_{\FT,K}^{\np,\circ,+}$. Let
	$$f(X,Y)\coloneqq \sum_{i,j\geq 0}\left(\lim_{n\to\infty}a_{ij}^{(n)}\right)X^iY^j.$$
	\Cref{lem:62550} ensures this element is well-defined and $\iota_{np}(f)$ is the limit of $\left\{\iota_{\np}(f_n)\right\}_{n\geq 1}$.
\end{proof}

In fact, we can use group-theoretic arguments to show:
\begin{proposition}\label{prop:56478}
	The field $\bfE_{\FT,K}^{\np ,\circ}$ is a proper subfield of $\bfE_{\FT,K}^{\np }$.
\end{proposition}
\begin{proof}
	We only need to show that $\rho$ is not an isomorphism. This can be done by comparing the rank (cf. \cite[339]{friedFieldArithmetic2008}) of these two groups.

	By \Cref{lem:1122}, $\bfE_{\FT,K}^{\np ,\circ}$ is isomorphic to $k_K\pparen{X,Y}$, which has absolute Galois group with rank equal to the cardinality of $k_K\pparen{X,Y}$ (cf. \cite[Theorem 1.1]{HARBATER2005623}). \cite[Remark 7.1 (iii)]{krapp2023generalised} tells us that this rank is uncountable.

	On the other hand, $H_{\frakF,K}$ is isomorphic to the absolute Galois of the field of norms $\bfE_{\FT,K}\cong k_K\pparen{T}$ by Fontaine-Wintenberger, which has countable rank (cf. \cite[Theorem 7.5.13]{Neukirch2008}).
\end{proof}

Although $\bfE_{\FT,K}^{\np ,\circ}$ is not satisfactory in the sense of \Cref{prop:56478}, we can still apply Zhao's argument on its completion $\widehat{\bfE}_{\FT,K}^{\np ,\circ}$ and obtain
\begin{lemma}\label{lem:47550}
	$\widehat{\rho}$ is an isomorphism, i.e., $\widehat{\bfE}_{\FT,K}^{\np ,\circ}=\widehat{\bfE}_{\FT,K}^{\np }$.
\end{lemma}
\begin{remark}
	When constructing $(\varphi,\tau)$-modules, one may pursue the minimal possible substitution for the perfectoid field $\wtilde{\bfE}_{\frakF,K}$. Such substitutions $F\subseteq \wtilde{\bfE}_{\frakF,K}$ must satisfy the following conditions:
	\begin{enumerate}[label=(T\arabic*)]
		\item $F$ contains $\bfE_{\tau,K}$ and $\bfE_K$. Here the inclusion $\bfE_K\subset F$ ensures that the $\tau$-action on $\bfE_{\tau,K}\subset F$ is well-defined.
		\item $\scrG_F$ is isomorphic to $H_{\frakF,K}$.
	\end{enumerate}
	It is easy to show that the field $\EtmKcnp$ is the minimal \underline{complete} subfield satisfying both conditions. On the other hand, $\EtmKnp$ is unlikely to be minimal if one drops the constraint of completeness. This is hinted by the fact that $\EtmKnpc$ is NOT the minimal subfield of $\Cpflat$ satisfying (T1) (cf. \cite[Proposition 1.8.]{Abhyankar1991}). It is still possible to find a smaller field than $\EtmKnp$ satisfying (T1) and (T2) simultaneously, but it may not have an effortless and explicit description.
\end{remark}
\begin{proposition}\label{prop:29251}
	The field $\bfE_{\FT,K}^{\np }$ is a proper dense subfield of $\EtmKcnp$.
\end{proposition}
\begin{proof}
	Since $\bfE_{\FT,K}^{\np }\subseteq \widehat{\bfE}_{\FT,K}^{\np }=\widehat{\bfE}_{\FT,K}^{\np ,\circ}$ by \Cref{def:8032} and \Cref{lem:47550}, we know that $\bfE_{\FT,K}^{\np }$ is dense in $\widehat{\bfE}_{\FT,K}^{\np }$. It remains to show that there exists a Cauchy sequence in $\bfE_{\FT,K}^{\np ,\circ}$ that converges to a transcendental element over $\bfE_{\FT,K}^{\np ,\circ}$.

	Take $\theta=p^r>\!\!\!> 0, r\in\bfZ_{\geq 1}$ such that $v^\flat\left(\frac{\eta_K^\theta}{u_K}\right)>0$. Let $f(X,Y)\coloneqq \sum_{n=0}^\infty Y^{n!}$ and
	$$A\coloneqq \iota_{\theta,X}(f)=\lim_{m\to\infty}\sum_{n=0}^m\left(\frac{\eta_K^\theta}{u_K}\right)^{n!}\in \EtmKcnp.$$
	If $A$ is separable (and consequentially algebraic) over $\EtmKnpc$, then there exist elements $a_0,\cdots,a_m$ in $\bfE_{\FT,K}^{\np,\circ,+}$ that %
	$a_0+a_1 A+\cdots+a_m A^m=0$. Since $\theta$ is a power of $p$, we can write
	\begin{equation}\label{eq:12584}
		a_0^{\theta}+a_1^{\theta} A^{\theta}+\cdots+a_m^{\theta}\left(A^{\theta}\right)^m=\left(a_0+a_1 A+\cdots+a_m A^m\right)^\theta=0.
	\end{equation}
	For every $s=0,\cdots,m$, write $a_s=\lim_{n\to\infty}\sum_{i=0}^n\sum_{j=0}^n a_{ij}^{(s)}u_K^i\eta_K^j$. Then
	\begin{align*}
		a_s^\theta= & \lim_{n\to \infty}\sum_{i=0}^n\sum_{j=0}^n \left(a_{ij}^{(s)}\right)^{\theta}u_K^{\theta i+j}\left(\frac{\eta_K^\theta}{u_K}\right)^j \\
		=           & \lim_{n\to \infty}\sum_{i=0}^{(\theta+1)n}\sum_{j=0}^{(\theta+1)n}c_{ij}u_K^i \left(\frac{\eta_K^\theta}{u_K}\right)^j                \\
		=           & \lim_{n\to \infty}\sum_{i=0}^n\sum_{j=0}^n c_{ij}u_K^i \left(\frac{\eta_K^\theta}{u_K}\right)^j,
	\end{align*}
	where
	$c_{ij}\coloneqq \begin{cases}
			\left(a_{\frac{i-j}{\theta},j}^{(s)}\right)^\theta, & \text{ if }i-j\in\theta\bfN; \\
			0,                                                  & \text{ otherwise}
		\end{cases}$.
	Therefore, we can write $a_s^\theta=\iota_{\theta,X}(h_s)$, where
	\begin{equation}\label{eq:20885}
		h_s(X,Y)=\sum_{i=0}^\infty\sum_{j=0}^\infty c_{ij}X^i Y^j=\sum_{i=0}^\infty \sum_{j=0}^\infty \left(a_{ij}^{(s)}\right)^\theta X^i (XY)^j.
	\end{equation}

	Thus, \Cref{eq:12584} can be rewritten as
	$$\sum_{s=0}^m \iota_{\theta,X}(h_s)\iota_{\theta,X}(f)^{\theta s}=0.$$
	By \Cref{lem:869}, this implies that
	$$\sum_{s=0}^m h_s(X,Y)f(X,Y)^{\theta s}=0.$$
	If we set $Z\coloneqq XY$, then by \Cref{eq:20885} the above equality can be rewritten as
	$$\sum_{s=0}^m \left(\sum_{i=0}^\infty \sum_{j=0}^\infty \left(a_{ij}^{(s)}\right)^\theta X^i Z^j\right)\left(\sum_{n=0}^\infty\left(\frac{Z}{X}\right)^{n!}\right)^{\theta s}=0.$$
	This implies the assertion, which is false by \cite[Example 7.2]{arocaSupportLaurentSeries2019}, that $\sum_{n=0}^\infty \left(\frac{Z}{X}\right)^{n!}$ is algebraic over $k_K\pparen{X,Z}$.
\end{proof}

\begin{proposition}\label{prop:35357}
	The valuation $v^\flat$ on $\EtmKnpc$ (resp. $\bfE_{\frakF,K}^{\np,?}$) is non-discrete.
\end{proposition}
\begin{proof}
	Since the residue field of $\bfE_{\tau,K}$ and $\wtilde{\bfE}_{\FT,K}$ are both $k$, the containing relationship $\bfE_{\tau,K}\subsetneq\EtmKcnp\subseteq \wtilde{\bfE}_{\FT,K}$ implies that the residue field of $\EtmKcnp$ is also $k$. If $v^\flat$ is discrete on $\EtmKnpc$, then $\EtmKcnp$ is a local field, i.e., a complete discrete valuation field with finite residue field. Thus,  $\EtmKcnp$ is isomorphic to $k_K\pparen{T}$ by the classification of local fields (cf. \cite[129]{serreLocalClassField1967}) and the $T$-adic topology on it coincides with the topology induced by $v^\flat$.

	Since $\bfE_{\tau,K}$ is a closed subfield of $\EtmKcnp$ with respect to the $T$-adic topology, a result of Nazir-Popescu (cf. \cite[Theorem 5]{zbMATH05243931}) tells us that $\EtmKcnp$ is a finite extension of $\bfE_{\tau,K}$, which is impossible since $\EtmKnpc\cong k_K\pparen{X,Y}\subsetneq \EtmKcnp$ is already a transcendental extension of $\bfE_{\tau,K}\cong k_K\pparen{X}$.
\end{proof}
\begin{corollary}
	The valuation on $\bfE_{\FT,K}^{\np ,\circ,+}$ is non-discrete.
\end{corollary}
\begin{proof}
	Use
	$$v^\flat\left(\bfE_{\FT,K}^{\np ,\circ}\right)=v^\flat\left(\bfE_{\FT,K}^{\np ,\circ,+}\right)-v^\flat\left(\bfE_{\FT,K}^{\np ,\circ,+}\right),$$
	where minus means Minkowski subtraction.
\end{proof}

\subsection{Imperfectness of $\bfE_{\FT,K}^{\np,?}$}
We finish this section by discussing the imperfectness of the period rings $\bfE_{\frakF,K}^{\np,?}$.
\begin{proposition}\label{prop:16686}
	For $\bfE\in\left\{\bfE_{\FT,K}^{\np ,\circ},\bfE_{\FT,K}^{\np,?},\bfE_{\FT}^{\np,?}\right\}$, we have $[\bfE\colon\varphi(\bfE)]=p^2$ and $\{u_K^i\cdot\eta_K^j\}_{0\leq i,j\leq p-1}$ is a $\varphi(\bfE)$-basis of $\bfE$. In particular, $\{u_K,\eta_K\}$ is a $p$-basis\footnote{A $p$-basis of a field $k$ is a subset $B$ of $k$, such that \begin{enumerate}
			\item $[k^p(b_1,\cdots,b_r)\colon k^p] = p^r$ for any $r$ distinct elements $b_1,\cdots,b_r\in B$;
			\item $k = k^p(B)$.
		\end{enumerate}} of $\bfE$.
\end{proposition}

The following lemma is a key input when proving this proposition (and also \Cref{lem:56945}) for the case of $\bfE_{\frakF,K}^{\np,?}=\EtmKcnp$:
\begin{lemma}
	Let $(E,w)/(F,v)$ be a finite extension of valued field of characteristic $p$. Denote by $\widehat{E}$ $\widehat{F}$ the completion of $E$ (resp. $F$) with respect to $w$ (resp. $v$). Then the following assertions hold:
	\begin{lemenum}
		\item\label{lem:39640} Any basis of this extension also generates $\widehat{E}$ as a $\widehat{F}$-vector space. In particular, one has $[\widehat{E}\colon\widehat{F}]\leq [E\colon F]$.
		\item\label{lem:16740} If $E/F$ is non-trivial and a generator $x\in E$ of it satisfies $x^p\in F$, then $E/F$ is a purely inseparable extension of degree $p$. In particular, one has $x\notin \widehat{F}$ under the same conditions.
	\end{lemenum}
\end{lemma}
\begin{proof}
	\hyperref[lem:39640]{(1)} is from \cite[64]{Warner_1986}.

	The first assertion of \hyperref[lem:16740]{(2)} can be found, for instance, in \cite[Corollary 19.11, Corollary 19.12]{isaacsAlgebraGraduateCourse1994}. To prove $x\notin \widehat{F}$, we first claim that $[\widehat{E}\colon\widehat{F}]=[E\colon F]$. If this claim holds, then the first assertion implies that the set $\{1,x,\cdots,x^{p-1}\}$, which is a basis of $E/F$, is also a basis of $\widehat{E}/\widehat{F}$, showing that $x\notin \widehat{F}$.

	In fact, the claim can be deduced from the following conditions (cf. \cite[Theorem 1(1)(5)]{Warner_1986}):
	\begin{enumerate}
		\item All valuation on $E$ that extends $v$ are equivalent. This follows from \cite[Corollary 3.2.10]{englerValuedFields2005} and the first assertion.
		\item There exist $c_0,c_1,\cdots,c_n\in E$ such that $F(c_0)$ is the separable closure of $F$ in $E$, $E=F(c_0,\cdots,c_n)$ and the subfields $\{F(c_0,\cdots,c_i)\}_{0\leq i\leq n}$, are closed in $E$. This can be verified by setting $c_i=x^i$, $i=0,1,\cdots,p-1$.
	\end{enumerate}
\end{proof}
\begin{proof}[Proof of \Cref{prop:16686}]
	We reformulate this proposition as follows:
	\begin{enumerate}[label=(A\arabic*),ref=A\arabic*]
		\item\label{it:49184} The element $u_K$ (resp. $\eta_K$) is not contained in $\varphi(\bfE)(\eta_K)$ (resp. $\varphi(\bfE)(u_K)$);
		\item\label{it:41340} $[\bfE\colon\varphi(\bfE)]\leq p^2$.
	\end{enumerate}
	Since $u_K^p\in\varphi(\bfE)(\eta_K)$ (resp. $\eta_K^p\in\varphi(\bfE)(u_K)$), \Cref{it:49184} ensures that $\{u_K^i\}_{0\leq i\leq p-1}$ (resp. $\{\eta_K^i\}_{0\leq i\leq p-1}$) are linearly independent over $\varphi(\bfE)(\eta_K)$ (resp. $\varphi(\bfE)(u_K)$). In particular, the extensions $\bfE/\varphi(\bfE)(\eta_K)$, $\varphi(\bfE)(\eta_K)/\varphi(\bfE)$, $\bfE/\varphi(\bfE)(u_K)$ and $\varphi(\bfE)(u_K)/\varphi(\bfE)$ have degree at least $p$. Together with \Cref{it:41340}, these 4 extensions all have degree $p$, and the result follows.

	For \Cref{it:49184}, we prove $u_K\notin \varphi(\bfE)$, and the remaining part can be proved similarly. The inclusion $\varphi(\bfE)(\eta_K)\subseteq \varphi\left(\widehat{\bfE}_{\FT}^{\np}\right)(\eta_K)$ reduces the problem to the case of $\bfE=\Ecsep$, i.e., $u_K\notin \varphi\left(\widehat{\bfE}_{\FT}^{\np}\right)(\eta_K)$. \Cref{lem:869} implies that $u_K$ is not contained in $\varphi\left(\EtmKnpc\right)(\eta_K)=k_K\pparen{u_K^p,\eta_K}$. By taking $E=k_K\pparen{u_K^p,\eta_K}(u_K)=\EtmKnpc$, $F=k_K\pparen{u_K^p,\eta_K}$ and $x=u_K$ in \Cref{lem:16740} which we will prove later, we can show that $u_K\notin \left(\yhwidehat{\varphi\left(\EtmKnpc\right)(\eta_K)}\right)^{\sep}$. Since $\yhwidehat{\varphi\left(\EtmKnpc\right)(\eta_K)}=\yhwidehat{\varphi\left(\EtmKnpc\right)}(\eta_K)$ by \cite[64]{Warner_1986}, we know that
	$$u_K\notin \left(\yhwidehat{\varphi\left(\EtmKnpc\right)}(\eta_K)\right)^{\sep}=\left(\varphi\left(\EtmKcnp\right)(\eta_K)\right)^{\sep}.$$
	Notice that $\varphi\left(\widehat{\bfE}_{\FT}^{\np}\right)(\eta_K)=\left(\varphi\left(\EtmKcnp\right)\right)^{\sep}(\eta_K)$ is contained in $\left(\varphi\left(\EtmKcnp\right)(\eta_K)\right)^{\sep}$, the result follows.

	Now we prove \Cref{it:41340}. The case of $\bfE=\EtmKnpc$ is immediate: $\{u_K^i\cdot\eta_K^j\}_{0\leq i,j\leq p-1}$ generates $\bfE_{\FT,K}^{\np,\circ,+}$ as a $\varphi\left(\bfE_{\FT,K}^{\np,\circ,+}\right)$-module and consequently generates $\EtmKnpc$ as $\varphi\left(\EtmKnpc\right)$-vector space. Since \Cref{lem:39640} indicates that $\{u_K^i\cdot\eta_K^j\}_{0\leq i,j\leq p-1}$ also generates $\EtmKcnp$ as $\varphi\left(\EtmKcnp\right)$-vector space, we can apply \cite[Lemma 3.3.1]{Zhao2022} to show that \Cref{it:41340} holds for $\bfE\in\left\{\Esep,\Ecsep\right\}$. In particular, the proof of \cite[Lemma 3.3.1]{Zhao2022} indicates that
	\begin{equation}\label{eq:25701}
		\varphi(\Esep)\EtmKnpc=\Esep.
	\end{equation}
	For $\bfE=\EtmKnp$, \cite[Theorem 26.4]{matsumuraCommutativeRingTheory1987} tells us that $\varphi(\Esep)=\varphi(\EtmKnp)^{\sep}$ and $\EtmKnp$ are linearly disjoint over $\varphi(\EtmKnp)$, i.e.,
	\begin{equation}\label{eq:40351}
		\varphi(\Esep)\otimes_{\varphi(\EtmKnp)}\EtmKnp\cong \varphi(\Esep)\EtmKnp.
	\end{equation}
	Since $\varphi(\Esep)\EtmKnpc\subset \varphi(\Esep)\EtmKnp\subset \Esep$, \eqref{eq:25701} implies that \eqref{eq:40351} can be rewritten as
	$$\varphi(\Esep)\otimes_{\varphi(\EtmKnp)}\EtmKnp\cong \Esep.$$
	As a result, we have $[\EtmKnp\colon\varphi(\EtmKnp)]=[\Esep\colon \varphi(\Esep)]\leq p^2$, and the proof is complete.
\end{proof}
\begin{corollary}\label{coro:42040}
	When $K=\bfQ_p$, $B=\{u_{\Qp},\varepsilon\}$ is a $p$-basis of $\bfE_{\FT,\Qp}^{\np,?}$ (resp. $\widehat{\bfE}_{\FT}^{\np,?}$).
\end{corollary}

It is natural to ask whether there is any containment relationship between $\bfE_{\frakF}$ and $\bfE_{\frakF,K}^?$. The imperfectness of $\bfE_{\frakF,K}^{\np,?}$ allows us to conclude the following result, which partially answers this question:
\begin{lemma}
	There exists an integer $n\geq 1$ such that $\varphi^{-n}\left(\bfE_{\FT,K}\right)$ is not contained in $\EtmKcnp$.
\end{lemma}
\begin{proof}
	If the statement is incorrect, then $\EtmKcnp$ contains $\bigcup_n\varphi^{-n}\left(\bfE_{\FT,K}\right)$ and consequentially its completion $\wtilde{\bfE}_{\FT,K}$. This implies that $\EtmKcnp=\wtilde{\bfE}_{\FT,K}$, which is impossible by \Cref{prop:16686}.
\end{proof}

\section{Liftability of the action of $\Gamma_{\FT,K}$}
To build the characteristic $0$ theory, the first step is to lift the action of the Galois group $\Gamma_{\FT,K}$ on the $\opn{mod} p$ period rings to their $p$-Cohen ring. We start by recalling the definition of the $p$-Cohen ring and the concept of liftability.
\begin{definition}\label{def:34646}
	For any field $k$ of characteristic $p$, the $p$-Cohen ring $\calC(k)$ of $k$ is the unique (up to isomorphism) absolutely unramified complete discrete valuation ring with residue field $k$.
\end{definition}
\begin{lemma}[{\cite[Section 3]{BSMF_1972__100__241_0}}]\label{lem:19332}
	For any field $k$ of characteristic $p$, let $B$ be a $p$-basis of $k$. For any $n\in\bfN$, let $\calC_{n+1}(k)$ be the subring of $W_{n+1}(k)$ generated by $W_{n+1}\left(k^{p^n}\right)$ and $\{[b]\colon b\in B\}$. Then the canonical projection $W_{n+1}(k)\lto W_n(k)$ restrict to a surjective morphism $\calC_{n+1}(k)\lto \calC_n(k)$, which induces an isomorphism $\calC(k)\cong \varprojlim_n\calC_n(k)$.
\end{lemma}
\begin{remark}
	The above construction of the $p$-Cohen ring depends on the choice of the $p$-basis $B$, which is not canonical.
\end{remark}

\begin{definition}
	Let $E$ be a subfield of $\wtilde{\bfE}_{\FT,K}$ satisfying that $E$ is stable under the action of $\scrG_K$ and this action factors through $\Gamma_{\FT,K}$. We call the action of $\Gamma_{\FT,K}$ on $E$ \textbf{liftable} if there exists a subring of $W\left(\wtilde{\bfE}_{\FT,K}\right)$ that satisfies
	\begin{enumerate}
		\item $R$ is a $p$-Cohen ring of $E$;
		\item $R$ is stable under the Frobenius map and the action of $\Gamma_{\FT,K}$ on $W\left(\wtilde{\bfE}_{\FT,K}\right)$:
		      $$\Gamma_{\FT,K}\times W\left(\wtilde{\bfE}_{\FT,K}\right)\lto W\left(\wtilde{\bfE}_{\FT,K}\right),\ g\cdot\sum_{n=0}^\infty[x_n]p^n\longmapsto \sum_{k=0}^\infty[g\cdot x_n]p^n.$$
	\end{enumerate}
\end{definition}

We are going to prove that the action of $\Gamma_{\FT,K}$ on $\bfE_{\frakF,K}^{\np,?}$ is liftable. The following lemma will be useful:
\begin{lemma}
	Let $K_1$ (resp. $K_1^\prime$) be the maximal tamely ramified extension of $K$ (resp. $\bfQ_p$) in $K_\infty$, and let $K_2$ (resp. $K_2^\prime$) be the maximal tamely ramified extension of $K$ (resp. $\bfQ_p$) in $K_{p^\infty}$. Then
	$K_1$ (resp. $K_2$) is a finite extension of $K_1^\prime$ (resp. $K_2^\prime$) and $[K_1\colon K_1^\prime]=[K_2\colon K_2^\prime]=p^{\frakF(K)}$ for some natural number $\frakF(K)$.
\end{lemma}
\begin{proof}
	By the definition of APF extension, the extensions $K_1/K$, $K_1^\prime/\Qp$, $K_2/K$ and $K_2^\prime/\Qp$ are all finite. Since $K_1^\prime K$ (resp. $K_2^\prime K$) is tamely ramified over $K$ (cf. \cite[Corollary 7.8]{Neukirch2008}), it follows that $K_1^\prime\subseteq K_1^\prime K\subseteq K_1$ (resp. $K_2^\prime\subseteq K_2^\prime K\subseteq K_2$) are finite extensions.

	As a consequence, $K_1^\prime$ is the maximal tamely ramified extension of $\Qp$ in $K_1$, which has ramification index coprime to $p$. If we view $K_1/K$ as finite APF extension, then \cite[Remarque 1.4.2]{Wintenberger1983} tells us that $[K_1\colon K_1^\prime]$ is a power of $p$, i.e.,
	$$[K_1\colon K_1^\prime]=\frake(K_1/K_1^\prime)=p^{v_p(\frake(K_1/\Qp))}.$$
	Similarly, we can show that
	$$[K_2\colon K_2^\prime]=\frake(K_2/K_2^\prime)=p^{v_p(\frake(K_2/\Qp))}.$$
	The important observation here is that $K_1$ actually coincides with $K$. Then one can write
	$$[K_2\colon K_2^\prime]=p^{v_p(\frake(K_2/K))+v_p(\frake(K/\Qp))}=p^{v_p(\frake(K_2/K))}\cdot [K_1\colon K_1^\prime]=[K_1\colon K_1^\prime],$$
	as desired.
\end{proof}
The following lemma helps us to reduce the problem of liftability to the case of $K=\Qp$:
\begin{lemma}\label{lem:36685}
	Let $K$ be a finite extension of $\Qp$.
	\begin{enumerate}
		\item The field $\bbE_{\frakF,K}^{\np,\circ}\coloneqq k_K\pparen{\pi_K^{p^{\frakF(K)}},\eta_K^{p^{\frakF(K)}}}=\varphi^{\frakF(K)}(\EtmKnpc)$ is a finite separable extension of $\bfE_{\frakF,\Qp}^{\np,\circ}$.%
		\item The field $\bbE_{\frakF,K}^{\np,?}\coloneqq\varphi^{\frakF(K)}\left(\bfE_{\frakF,K}^{\np,?}\right)$ is a finite separable extension of $\bfE_{\frakF,\Qp}^{\np,?}$.
	\end{enumerate}
\end{lemma}
\begin{proof}
	Consider the field of norms $X_{\Qp}(K_\infty)$, which is a finite separable extension of $X_{\bfQ_p}(\bfQ_{p,\infty})$ by \cite[Théorème 3.1.2]{Wintenberger1983}. The identification
	$$X_{\Qp}(K_\infty)\xlongrightarrow{\cong} X_K(K_\infty),\ (\alpha_E)_{E\in\scrE(K_\infty/\Qp)}\longmapsto (\alpha_E)_{E\in\scrE(K_\infty/K)}$$
	becomes a purely inseparable extension of degree $p^{\frakF(K)}$ if we embed them into $\Cpflat$ by using the formula in \Cref{thm:55903}. If we denote by $\bbE_{\tau,K}$ the image of $X_{\Qp}(K_\infty)$ in $\Cpflat$, then one can write $\bbE_{\tau,K}=k_K\pparen{u_K^\prime}$ with $u_K^\prime=u_K^{p^{\frakF(K)}}$, i.e., $\bbE_{\tau,K}=\varphi^{\frakF(K)}(\bfE_{\tau,K})$. Similarly, if we denote by $\bbE_{K}$ the image of $X_{\Qp}(K_{p^\infty})$ in $\Cpflat$, then $\bbE_K=k_K\pparen{\eta_K^\prime}$ with $\eta_K^\prime=\eta_K^{p^{\frakF(K)}}$ is a finite separable extension of $\bfE_{\Qp}$ and $\bbE_K=\varphi^{\frakF(K)}(\bfE_K)$.

	\begin{enumerate}
		\item The first assertion is equivalent to prove that $k_K\pparen{u_K^\prime,\eta_K^\prime}$ is a finite separable extension of $\bfE_{\frakF,\Qp}^{\np,\circ}$.

		      Take a primitive element $\alpha$ of the extension $\bfE_{\tau,K}/\bfE_{\tau,\Qp}$, and we fix a set of generators $\{s_1(\alpha),\cdots,s_d(\alpha)\}$ of $k_K\bbrac{u_K^\prime}$ as a finite free $\bfF_p\bbrac{u_{\Qp}}$-module. Then
		      $$k_K\bbrac{u_K^\prime,\eta_K^\prime}=k_K\bbrac{u_K^\prime}\bbrac{\eta_K^\prime}=\left(\sum_{i=1}^d k_K\bbrac{u_{\Qp}}s_i(\alpha)\right)\bbrac{\eta_K^\prime}=\sum_{i=1}^d\left(k_K\bbrac{u_{\Qp}}\bbrac{\eta_K^\prime}\right)s_i(\alpha),$$
		      which lies in $k_K\pparen{u_{\Qp},\eta_K^\prime}(\alpha)$. Since $\alpha$ is separable over $\bfE_{\tau,K}\subset k_K\pparen{u_{\Qp},\eta_K^\prime}$, one obtains that $k_K\pparen{u_K^\prime,\eta_K^\prime}=\opn{Frac}k_K\bbrac{u_K^\prime,\eta_K^\prime}$ is finite separable over $k_K\pparen{u_{\Qp},\eta_K^\prime}$. With the same strategy, one can show that $k_K\pparen{u_{\Qp},\eta_K^\prime}$ is finite separable over $\bfE_{\frakF,\Qp}^{\np,\circ}$, which completes the proof of the first assertion.
		\item For $\bfE_{\frakF,K}^{\np,?}=\bfE_{\frakF,K}^{\np}$, since $\bbE_{\frakF,K}^{\np,\circ}$ is finite separable over $\bfE_{\frakF,\Qp}^{\np,\circ}$ by the first assertion, we know that
		      $$\left(\bfE_{\frakF,K}^{\np,\circ}\right)^{\opn{sep}}=\left(\bbE_{\frakF,K}^{\np,\circ}\right)^{\opn{sep}}=\varphi^{\frakF(K)}\left(\left(\bfE_{\frakF,K}^{\np,\circ}\right)^{\opn{sep}}\right).$$
		      Since $\bfE_{\frakF,\Qp}^{\np}=\left(\left(\bfE_{\frakF,\Qp}^{\np,\circ}\right)^{\opn{sep}}\right)^{H_{\frakF,\Qp}}$,
		      $$\varphi^{\frakF(K)}\left(\EtmKnp\right)=\varphi^{\frakF(K)}\left(\left(\left(\bfE_{\frakF,K}^{\np,\circ}\right)^{\opn{sep}}\right)^{H_{\frakF,K}}\right)=\left(\left(\bbE_{\frakF,K}^{\np,\circ}\right)^{\opn{sep}}\right)^{H_{\frakF,K}}$$
		      and $H_{\frakF,K}$ is a finite index subgroup of $H_{\frakF,\Qp}$, the result follows.

		      For $\bfE_{\frakF,K}^{\np,?}=\EtmKcnp$,	we take a primitive element $\beta$ of the extension $\bbE_{\frakF,K}^{\np,\circ}/\bfE_{\frakF,\Qp}^{\np,\circ}$. By \Cref{lem:39640}, the set $\{1,\beta,\cdots,\beta^{e-1}\}$, where $e=[\bbE_{\frakF,K}^{\np,\circ}\colon\bfE_{\frakF,\Qp}^{\np,\circ}]$, is a generating set (which is not necessarily a basis!) of
		      $$\widehat{\bbE_{\frakF,K}^{\np,\circ}}=\widehat{\varphi^{\frakF(K)}(\EtmKnpc)}=\varphi^{\frakF(K)}(\EtmKcnp)=\widehat{\bbE}_{\frakF,K}^{\np,\circ}$$
		      as a $\widehat{\bfE}_{\frakF,\Qp}^{\np,\circ}$-vector space. This implies that $\widehat{\bbE}_{\frakF,K}^{\np,\circ}\subseteq \widehat{\bfE}_{\frakF,\Qp}^{\np,\circ}(\beta)$, where $\widehat{\bfE}_{\frakF,\Qp}^{\np,\circ}(\beta)$ is finite separable over $\widehat{\bfE}_{\frakF,\Qp}^{\np,\circ}\supset \bfE_{\frakF,\Qp}^{\np,\circ}$.
	\end{enumerate}
\end{proof}

\begin{proposition}\label{prop:42040}
	The action of $\Gamma_{\frakF,K}$ on $\bfE_{\frakF,K}^{\np,?}$ is liftable.
\end{proposition}
\begin{proof}
	\begin{enumerate}
		\item We first prove the case of $K=\bfQ_p$. By \Cref{coro:42040}, we may take the $p$-basis of $\bfE_{\frakF,\Qp}^{\np,?}$ to be $\{u_{\Qp},\varepsilon\}$. Thus, by \Cref{lem:19332}, it is reduced to prove that $\calC_{n+1}\left(\bfE_{\frakF,K}^{\np,?}\right)=W_{n+1}\left(\left(\bfE_{\frakF,K}^{\np,?}\right)^{p^n}\right)([u_{\Qp}],[\varepsilon])$ is stable under the action of $\Gamma_{\FT,\Qp}$ on $W\left(\wtilde{\bfE}_{\FT,\Qp}\right)$. This can be checked directly: $W_{n+1}\left(\left(\bfE_{\frakF,K}^{\np,?}\right)^{p^n}\right)$ is stable under the action of $\Gamma_{\FT,\Qp}$ since $\left(\bfE_{\frakF,K}^{\np,?}\right)^{p^n}$ is, and $\Gamma_{\FT,\Qp}$ acts on $[u_{\Qp}]$ and $[\varepsilon]$ by the formula
		      $$\gamma_{\Qp}\cdot [u_{\Qp}]=[u_{\Qp}], \tau_{\Qp}\cdot [u_{\Qp}]=[\varepsilon]\cdot [u_{\Qp}]$$
		      and
		      $$\gamma_{\Qp}\cdot [\varepsilon]=[\varepsilon]^{\chi(\gamma_{\Qp})},\ \tau_{\Qp}\cdot [\varepsilon]=[\varepsilon].$$
		      We denote by $\bfA_{\frakF,\bfQ}^{\np,?}$ the constructed $p$-Cohen ring of $\bfE_{\frakF,\Qp}^{\np,?}$ in $W\left(\wtilde{\bfE}_{\frakF,\Qp}\right)$.
		\item For general $p$-adic field $K$, we first show that the $\Gamma_{\frakF,K}$-action on $\bbE_{\frakF,K}^{\np,?}$ is liftable. \Cref{lem:36685} tells us that $\bbE_{\frakF,K}^{\np,?}$ is a finite separable extension of $\bfE_{\frakF,\Qp}^{\np,?}$. Therefore, there exists an irreducible separable polynomial $\overline{w}(T)\in \bfE_{\frakF,\Qp}^{\np,?}[T]$ that $$\bbE_{\frakF,K}^{\np,?}=\bfE_{\frakF,\Qp}^{\np,?}(T)/\overline{w}(T).$$ Take a lift $u(T)$ of $\overline{u}(T)$ in $\bfA_{\frakF,\bfQ}^{\np,?}[T]$. By Hensel's lemma, $u(T)$ has a root $\theta_K$ in $W\left(\wtilde{\bfE}_{\frakF,K}\right)$. We define $\bbA_{\frakF,K}^{\np,?}\coloneqq \bfA_{\frakF,\Qp}^{\np,?}(\theta_K)$, which is a finite unramified ring extension of $\bfA_{\frakF,\Qp}^{\np,?}$ with residue field $\bbE_{\frakF,K}^{\np,?}$. Then $\bbA_{\frakF,K}^{\np,?}$ is a $p$-Cohen ring of $\bbE_{\frakF,K}^{\np,?}$ in $W\left(\wtilde{\bfE}_{\frakF,K}\right)$. To prove that $\bbA_{\frakF,K}^{\np,?}$ is stable under the Frobenius map $F$ and the action of $\Gamma_{\FT,K}$ on $W\left(\wtilde{\bfE}_{\frakF,K}\right)$, it suffices to show that $F(\theta_K)\in \bbA_{\frakF,K}^{\np,?}$ and $g(\theta_K)\in \bbA_{\frakF,K}^{\np,?}$ for every $g\in \Gamma_{\FT,K}$. Since $F(\theta_K)$ (resp. $g(\theta_K)$) is a root of the polynomial $F(u)(T)$ (resp. $g(u)(T)$), whose reduction $\varphi(\overline{u})(T)$ (resp. $g(\overline{u})(T)$) splits in $\bbE_{\frakF,K}^{\np,?}[T]$, the result follows from Hensel's lemma again.

		      Finally, we define
		      $$\bfA_{\frakF,K}^{\np,?}\coloneqq \left\{x\in W\left(\wtilde{\bfE}_{\frakF,K}\right)\colon F^{\frakF(K)}(x)\in \bbA_{\frakF,K}^{\np,?}\right\}.$$
		      Since $\bfA_{\frakF,K}^{\np,?}$ is isomorphic to $\bbA_{\frakF,K}^{\np,?}$ as $\bfE_{\frakF,K}^{\np,?}\cong\bbE_{\frakF,K}^{\np,?}$, $\bfA_{\frakF,K}^{\np,?}$ is automatically a $p$-Cohen ring of $\bfE_{\frakF,K}^{\np,?}$. To verify the liftability of $\Gamma_{\frakF,K}$ on $\bfA_{\frakF,K}^{\np,?}$, we only need to write
		      $$\bfA_{\frakF,K}^{\np,?}=\left\{\sum_{n}\left[a_n^{1/p^{\frakF(K)}}\right]p^n\in W\left(\wtilde{\bfE}_{\frakF,K}\right)\colon \sum_n [a_n]p^n\in \bbA_{\frakF,K}^{\np,?}\right\}$$\
		      and apply the liftability of $\Gamma_{\frakF,K}$ on $\bbA_{\frakF,K}^{\np,?}$.
	\end{enumerate}
\end{proof}
\begin{remark}
	It makes no harm if we replace $\EtmKnpc$ by $\bbE_{\frakF,K}^{\np,\circ}$ everywhere and rebuild the whole theory in this paper. We prefer to use $\EtmKnpc$, for most definitions of $(\varphi,\tau)$-modules or Breuil-Kisin modules in the literatures directly use the ``absolute'' field of norms $X_K(K_\infty)\cong\bfE_{\tau,K}=k_K\pparen{u_K}$ instead of the ``relative-to-$\Qp$'' field of norms $X_{\Qp}(K_\infty)\cong\bbE_{\tau,K}$ as the period ring.
\end{remark}

\begin{definition}\label{def:325}\leavevmode
	\begin{enumerate}
		\item Let $\bfA_{\FT,\Qp}^{\np}$ and $\bfA_{\FT,\Qp}^{\np,\opn{c}}$ be the $p$-Cohen ring of $\bfE_{\FT,\Qp}^{\np}$ and $\widehat{\bfE}_{\FT,\Qp}^{\np}$ respectively. We regard them as a subring of $W\left(\wtilde{\bfE}_{\FT,\Qp}\right)$ with respect to the $p$-basis $\{u_{\Qp},\varepsilon\}$ by \Cref{lem:19332}.
		\item Let $\bfA_{\FT}^{\np}$ (resp. $\bfA_{\FT}^{\np,\opn{c}}$) be the completion of the maximal unramified ring extension of $\bfA_{\FT,\Qp}^{\np}$ (resp. $\bfA_{\FT,\Qp}^{\np,\opn{c}}$), which is also the Cohen ring of $\Esep$ (resp. $\Ecsep$).
	\end{enumerate}
\end{definition}
\begin{remark}
	To be consistent with \Cref{rmk:84309}, the corresponding Cohen ring of $\bfE_{\frakF,K}^{\np,?}$ (resp. $\bfE_{\frakF}^{\np}$) in \Cref{def:325} will be denoted as $\bfA_{\frakF,K}^{\np,?}$ (resp. $\bfA_{\frakF}^{\np,?}$).
\end{remark}
By applying dévissage arguments on \Cref{lem:5466} and \Cref{coro:42040}, we obtain:
\begin{lemma}\leavevmode
	\begin{lemenum}
		\item $\left(\bfA_{\FT}^{\np,?}\right)^{H_{\frakF,K}}=\bfA_{\FT,K}^{\np,?}$.
		\item\label{coro:38666} When $K=\Qp$, $\left\{[u_{\Qp}]^i\cdot[\varepsilon]^j\right\}_{0\leq i,j\leq p-1}$ is a basis of $\bfA_{\FT}^{\np,?}$ (resp. $\bfA_{\FT,\Qp}^{\np,?}$) over $\varphi\left(\bfA_{\FT}^{\np,?}\right)$ (resp. $\varphi\left(\bfA_{\FT,\Qp}^{\np,?}\right)$).
	\end{lemenum}
\end{lemma}

\section{Bridging $(\varphi,\Gamma)$-modules and $(\varphi,\tau)$-modules}
The goal of this section is to bridge the $(\varphi,\Gamma)$-modules and $(\varphi,\tau)$-modules over imperfect period rings via $(\varphi,\Gamma_{\frakF,K})$-modules.
\subsection{Trivial case: the period rings}
We first compare the imperfect period rings of $(\varphi,\Gamma)$-modules, $(\varphi,\tau)$-modules and $(\varphi,\Gamma_{\frakF,K})$-modules.
The following result is a generalization of \cite[Proposition 1.9]{Caruso2013}:
\begin{proposition}\label{lem:56945}
	For $E\in\left\{\EtmKnpc, \bfE_{\FT,K}^{\np,?}\right\}$, we have $E^{H_{\tau,K}}=\bfE_{\tau,K}$ and $E^{H_K}=\bfE_K$.
\end{proposition}

\begin{proof}
	We first prove $E^{H_{\tau,K}}=\bfE_{\tau,K}$.
	\begin{enumerate}[wide]
		\item\label{it:1724} \textbf{Case of $\left(\EtmKnpc\right)^{H_{\tau,K}}$.} %
		      The theory of the field of norms gives us $\left(\EtmKnpc\right)^{H_{\tau,K}}\supseteq\bfE_{\tau,K}$. Since $H_{\frakF,K}$ acts trivially on $\left(\EtmKnpc\right)^{H_{\tau,K}}$, we only need to prove
		      $\left(\EtmKnpc\right)^{\gamma=1}\subseteq\bfE_{\tau,K}$ (cf. \Cref{lem:22831}). %

		      Let $x=\frac{f\left(u_K,\eta_K\right)}{g\left(u_K,\eta_K\right)}\in\left(\EtmKnpc\right)^{\gamma=1}$, where $f,g\in k_K\bbrac{X,Y}$. Then by \Cref{lem:13864}, the condition $x=\gamma\cdot x$ is equivalent to
		      \begin{equation}\label{eq:51283}
			      \frac{f\left(u_K,\eta_K\right)}{g\left(u_K,\eta_K\right)}=\frac{f\left(u_K, \eta_K\cdot \frakG_K\left(\eta_K\right)\right)}{g\left(u_K, \eta_K\cdot \frakG_K\left(\eta_K\right)\right)}.
		      \end{equation}
		      By \Cref{lem:1122}, this is equivalent to
		      \begin{equation}\label{eq:25993}
			      \frac{f(X,Y)}{g(X,Y)}=\frac{f(X, Y\cdot\frakG_K(Y))}{g(X,Y\cdot \frakG_K(Y))}\in k_K\pparen{X,Y}.
		      \end{equation}
		      If we view $f/g$ as an element of $k_K\pparen{Y}\pparen{X}$ and let its $X$-adic expansion be $f/g=\sum_{n=l}^\infty h_n(Y)X^n$, where $h_n(Y)\in k_K\pparen{Y}$, then \Cref{eq:25993} implies
		      \begin{equation}\label{eq:55536}
			      \sum_{n=l}^\infty h_n(Y)X^n=\sum_{n=l}^\infty h_n(Y\cdot \frakG_K(Y))X^n,
		      \end{equation}
		      i.e., $h_n(Y)=h_n(Y\cdot \frakG_K(Y))$ for all $n\neq 0$.

		      Notice that the condition $h_n(Y)=h_n(Y\cdot \frakG_K(Y))$ is equivalent to
		      $h_n\left(\eta_K\right)\in \bfE_K^{\gamma=1}$. \cite[Lemma 3.3.19]{Zhao2022} tells us that $\bfE_K^{\gamma=1}=k_K$, therefore we have $h_n(Y)\in k_K$ for every $n\neq 0$. This shows that $f/g\in k_K\pparen{X}$, i.e., $x\in \bfE_{\tau,K}$.
		\item \textbf{Case of $\left(\EtmKnp\right)^{H_{\tau,K}}$.} \Cref{it:1724} gives the inclusion $\bfE_{\tau,K}\subseteq \left(\EtmKnp\right)^{H_{\tau,K}}$. Suppose $y$ is an element of $\left(\EtmKnp\right)^{H_{\tau,K}}$ and let $r(T)$ be the minimal polynomial of it over $\EtmKnpc$. Since $y$ is invariant under the action of $H_{\tau,K}$, we have $y=\gamma\cdot y$ and consequentially
		      $$0=\gamma\cdot (r(y))=(\gamma\cdot r)(\gamma\cdot y)=(\gamma\cdot r)(y),$$
		      i.e., $(\gamma\cdot r)(T)$ is also the minimal polynomial of $y$ over $\EtmKnpc$. Thus the coefficients of $r(T)$ is fixed by the action of $H_{\tau,K}$ and belong to $\left(\EtmKnpc\right)^{H_{\tau,K}}=\bfE_{\tau,K}$. Therefore $y\in \left(\bfE_{\tau,K}^{\operatorname{sep}}\right)^{H_{\tau,K}}=\bfE_{\tau,K}$.%

		\item \textbf{Case of $\left(\EtmKcnp\right)^{H_{\tau,K}}$.}
		      The statement $\left(\EtmKcnp\right)^{H_{\tau,K}}=\bfE_{\tau,K}$ is equivalent to the assertion that $\EtmKcnp\cap\wtilde{\bfE}_{\tau,K}=\bfE_{\tau,K}$. The inclusion $\bfE_{\tau,K}\subseteq \EtmKcnp\cap\wtilde{\bfE}_{\tau,K}$ is trivial.

		      Suppose that $x\in\wtilde{\bfE}_{\tau,K}\backslash\bfE_{\tau,K}$. We claim that $x$ is not an element of $\EtmKnpc(x^p)$. If this is not the case, then there exist $f(T),g(T)\in\EtmKnpc[T]$ with $\gcd(f,g)=1$, such that $x$ is a root of the polynomial $h(T)\coloneqq f(T^p)-T\cdot g(T^p)$. Since $h^\prime(T)=-g(T^p)$, we know that $\gcd(h(T),h^\prime(T))=\gcd(f(T^p),g(T^p))=1$, i.e., $h(T)$ is separable. Therefore, $x$ belongs to
		      $$\left(\EtmKnpc\right)^{\opn{sep}}\cap\wtilde{\bfE}_{\tau,K}=\left(\left(\EtmKnpc\right)^{\opn{sep}}\right)^{H_{\tau,K}}=\left(\EtmKnp\right)^{H_{\tau,K}}=\bfE_{\tau,K},$$
		      which is a contradiction.

		      By the claim, we can set $E=\EtmKnpc(x^p)(x)=\EtmKnpc(x)$ and $F=\EtmKnpc(x^p)$ in \Cref{lem:16740} to conclude that $x$ is not contained in $\widehat{\EtmKnpc(x^p)}$. In particular, $x$ is not contained in $\EtmKcnp\cap\wtilde{\bfE}_{\tau,K}$. This completes the proof.
	\end{enumerate}

	The assertion $E^{H_K}=\bfE_K$ can be deduced similarly. Thus, we only give the brief proof of $\left(\EtmKnpc\right)^{\tau=1}\subseteq \bfE_K$, where $\tau=\tau_K$ is the topological generator of $\wtilde{\Gamma}_{\tau,K}$ in \Cref{lem:13864}.%

	Let $x=\frac{f\left(u_K,\eta_K\right)}{g\left(u_K,\eta_K\right)}\in\left(\EtmKnpc\right)^{\tau=1}$, where $f,g\in k_K\bbrac{X,Y}$, Then the condition $x=\tau\cdot x$ translates to
	\begin{equation}\label{eq:17658}
		\frac{f(X,Y)}{g(X,Y)}=\frac{f(X(1+Y\cdot\frakT_K(Y)),Y)}{g(X(1+Y\cdot\frakT_K(Y)),Y)}\in \kappa_K\pparen{X,Y}.
	\end{equation}
	by \Cref{lem:1122}.
	If we view $f/g$ as an element of $k_K\pparen{Y}\pparen{X}$ and let the $X$-adic expansion of it be $f/g=\sum_{n=l}^\infty h_n(Y)X^n$, where $h_n(Y)\in k_K\pparen{Y}$, then \Cref{eq:17658} is equivalent to
	$$\sum_{n=l}^\infty h_n(Y)X^n=\sum_{n=l}^\infty h_n(Y)\cdot(1+Y\cdot\frakT_K(Y))^nX^n.$$
	This implies that $h_n(Y)=0$ for all $n\neq 0$. Therefore $x=h_0\left(\eta_K\right)\in\bfE_K$.
\end{proof}
\begin{remark}\label{rmk:304}
	When Caruso proved $\left(\bfE_{\FT,K}^{\mathtt{Car},\np ,\circ}\right)^{H_{\tau,K}}=\bfE_{\tau,K}$ in \cite[Proposition 1.9]{Caruso2013}, he embedded the field $\bfE_{\FT,K}^{\mathtt{Car},\np ,\circ}$ into ``$k_K\pparen{u_K}\pparen{\eta_{\Qp}}$'' without applying the embedding lemma explicitly.  Although his proof is technically spotless, it is unclear whether one can extend $\iota_{\mathtt{Car}}$ to $k_K\pparen{X}\pparen{Y}$. Even if this is realizable, elements in the image of this extended embedding can not be written in the form
	$$\sum_{n>-\infty}h_n(u_K)\eta_{\Qp}^n,\ h_n\in k_K\pparen{Y},$$
	since $u_K$ has positive valuation and summations like $\sum_{n=1}^\infty \left(\frac{\eta_{\Qp}}{u_K^\theta}\right)^n,\ \theta>\!\!\!>0$ do not converge with respect to $v^\flat$.

	This subtlety brings much difficulty to the study of multivariable period rings. To do most calculations, we have to use the embedding lemma to pull $\bfE_{\FT,K}^{\np ,\circ}$ back to $k_K\pparen{X,Y}$, which loses most of the topological information.
\end{remark}

\begin{corollary}\label{coro:5386}
	One has $\left(\bfA_{\FT,K}^{\np,?}\right)^{H_{\tau,K}}=\bfA_{\tau,K}$ and $\left(\bfA_{\FT,K}^{\np,?}\right)^{H_K}=\bfA_K$.
\end{corollary}
\begin{proof}
	By dévissage and \Cref{lem:56945}.
\end{proof}

\subsection{General case: the $(\varphi,\Gamma_{\FT,K})$-modules}
\begin{definition}\label{def:44559}
	By replacing $\bfA_{\frakF,K}^{\ttR}$ with $\bfA_{\frakF,K}^{\np?}$ in the definition of \'etale $(\varphi,\Gamma_{\frakF,K})$-modules over $\bfA_{\frakF,K}^{\ttR}$ in \Cref{eg:5636}, one defines \'etale $(\varphi,\Gamma_{\frakF,K})$-modules over $\bfA_{\frakF,K}^{\np,?}$.
\end{definition}

It is immediate to generalize \Cref{thm:36449} for \'etale $(\varphi,\Gamma_{\frakF,K})$-modules over $\bfA_{\frakF,K}^{\np,?}$:
\begin{proposition}
	For any $V\in\repcat$, let $\bfD_{\FT}^{\np,?}(V)\coloneqq\left(\bfA_{\FT}^{\np,?}\otimes_{\Zp}V\right)^{H_{\frakF,K}}$. Then the functor $V\longmapsto \bfD_{\FT}^{\np,?}(V)$ induces an equivalence of categories
	$$\repcat\lto\left\{\text{\'etale }(\varphi,\Gamma_{\FT,K})\hyphen\text{modules over }\bfA_{\FT,K}^{\np,?} \right\},$$
	with the quasi-inverse given by $D\longmapsto \bfV(D)\coloneqq \left(\bfA_{\FT}^{\np,?}\otimes_{\bfA_{\FT,K}^{\np,?}}D\right)^{\varphi=1}$.
\end{proposition}

Notice that if we take $V$ to be the trivial $\bfZ_p$-representation of $\scrG_K$, then \Cref{coro:5386} indicates that $\bfD_{\FT}^{\np,?}(V)^{H_K}\cong \bfD(V)$ and $\bfD_{\FT}^{\np,?}(V)^{H_{\tau,K}}\cong \bfD_{\tau}(V)$. The following proposition generalizes this observation to arbitrary $\bfZ_p$-representations of $\scrG_K$:
\begin{proposition}\label{prop:48076}
	For any $V\in\repcat$, one has
	\begin{enumerate}
		\item $\bfA_{\FT,K}^{\np,?}\otimes_{\bfA_K}\bfD(V)\cong \bfD_{\FT}^{\np,?}(V)$ and
		      $\bfD_{\FT}^{\np,?}(V)^{H_K}\cong \bfD(V)$, where $\bfD(V)$ is the \'etale $(\varphi,\Gamma)$-module over $\bfA_K$ associated to $V$;
		\item $\bfA_{\FT,K}^{\np,?}\otimes_{\bfA_{\tau,K}}\bfD_\tau^{\np,?}(V)\cong \bfD_{\FT}^{\np,?}(V)$ and
		      $\bfD_{\FT}^{\np,?}(V)^{H_{\tau,K}}\cong \bfD_{\tau}^{\np,?}(V)$, where $\bfD_\tau^{\np,?}(V)$ is the \'etale $(\varphi,\tau)$-module over $\left(\bfA_{\tau,K},\bfA_{\FT,K}^{\np,?}\right)$ associated to $V$.
	\end{enumerate}
\end{proposition}
\begin{proof}
	We only prove the first assertion. The second one can be proved similarly.

	Let $\bfA$ be the $p$-adic completion of the maximal unramified ring extension of $\bfA_K$ in $\bfA_{\inf}$. By the well-known isomorphism of Fontaine (cf. \cite[Proposition A.1.2.4]{Fontaine1990})
	$$\bfA\otimes_{\bfA_K}\bfD(V)\cong \bfA\otimes_{\Zp}V,$$
	one has
	\begin{align*}
		\bfD_{\FT}^{\np,?}(V)= & \left(\bfA_{\FT}^{\np,?}\otimes_{\bfA}\left(\bfA\otimes_{\Zp}V\right)\right)^{\scrG_{\bfQ_{p,\FT}}}  \\
		\cong                  & \left(\bfA_{\FT}^{\np,?}\otimes_{\bfA}\left(\bfA\otimes_{\bfA_K}\bfD(V)\right)\right)^{H_{\frakF,K}} \\
		\cong                  & \left(\bfA_{\FT}^{\np,?}\otimes_{\bfA_K}\bfD(V)\right)^{H_{\frakF,K}}.
	\end{align*}
	Since $H_{\frakF,K}$ acts trivially on $\bfD(V)$ and $\bfA_K$, we obtain
	$$\bfD_{\FT}^{\np,?}(V)\cong \left(\bfA_{\FT}^{\np,?}\right)^{H_{\frakF,K}}\otimes_{\bfA_K}\bfD(V)=\bfA_{\FT,K}^{\np,?}\otimes_{\bfA_K}\bfD(V).$$
	The assertion $\bfD_{\FT}^{\np,?}(V)^{H_K}\cong \bfD(V)$ then follows from \Cref{coro:5386}.
\end{proof}
\begin{corollary}\label{coro:3935}
	The functor
	$$\left\{\text{\'etale }(\varphi,\Gamma)\hyphen\text{modules over }\bfA_K\right\}\lto\left\{\text{\'etale }(\varphi,\tau)\hyphen\text{modules over }\left(\bfA_{\tau,K},\bfA_{\FT,K}^{\np,?}\right)\right\},$$
	$$D\longmapsto \left(\bfA_{\FT,K}^{\np,?}\otimes_{\bfA_K}D\right)^{H_{\tau,K}}$$
	induces the equivalence of categories, with the quasi-inverse given by
	$$D^\prime\longmapsto \left(\bfA_{\FT,K}^{\np,?}\otimes_{\bfA_{\tau,K}}D^\prime\right)^{H_K}.$$%
\end{corollary}

\section{Galois cohomology via \'etale $(\varphi,\Gamma_{\frakF,K})$-modules}\label{sec:6180}

\subsection{The cohomology of $\opn{Kos}\left(\varphi,\Gamma_{\frakF,K},\bfD_{\frakF}^{\np,?}(V)\right)$}
By replacing the $(\varphi,\tau)$-module $\wtilde{\bfD}_\tau(V)$ (resp. the $(\varphi,\Gamma_{\frakF,K})$-module $\bfD_{\frakF}^{\ttR}(V)$) in the Herr-Zhao complex (resp. the Herr-Ribeiro complex) by $\bfD_{\tau}^{\np,?}(V)$ (resp. $\bfD_{\frakF}^{\np,?}(V)$), we get the complexes $\opn{Kos}\left(\varphi,\tau,\bfD_{\tau}^{\np,?}(V)\right)$ and $\opn{Kos}\left(\varphi,\Gamma_{\frakF,K},\bfD_{\frakF}^{\np,?}(V)\right)$ for these modules over the imperfect period rings. With identical arguments as in \cite{Zhao2022} and \cite{ribeiroExplicitFormulaHilbert2011}, one can prove the following proposition:
\begin{proposition}\label{thm:52100}
	One has quasi-isomorphisms
	$$\kappa_\infty^*\colon\opn{Kos}\left(\varphi,\Gamma_{\frakF,K},\bfD_{\tau}^{\np,?}(V)\right)\cong \rmR\Gamma_{\opn{cont}}(\scrG_K,V)$$
	and
	$$\kappa_{\frakF}^*\colon\opn{Kos}\left(\varphi,\wtilde{\Gamma}_{\tau,K},\bfD_{\frakF}^{\np,?}(V)\right)\cong \rmR\Gamma_{\opn{cont}}(\scrG_K,V).$$
\end{proposition}

On the other hand, Zhao proved in \cite[Section 2]{Zhao2022} that there exists a natural quasi-isomorphism
$$\opn{Kos}\left(\varphi,\Gamma_{\frakF,K},W\left(\wtilde{\bfE}_{\frakF,K}\right)\otimes_{\bfA_{\tau,K}}\wtilde{\bfD}_{\tau}(V)\right)\xlongrightarrow{\cong}\opn{Kos}\left(\varphi,\wtilde{\Gamma}_{\tau,K},\wtilde{\bfD}_{\tau}(V)\right).$$
Motivated by this quasi-isomorphism and \Cref{prop:48076}, we want to know how the complexes $\opn{Kos}\left(\varphi,\Gamma_K,\bfD(V)\right)$, $\opn{Kos}\left(\varphi,\Gamma_{\FT,K},\bfD_{\FT}^{\np,?}(V)\right)$ and $\opn{Kos}\left(\varphi,\wtilde{\Gamma}_{\tau,K},\bfD_\tau^{\np,?}(V)\right)$ of these modules over imperfect period rings can be related without the bridge of $\rmR\Gamma_{\opn{cont}}\left(\scrG_K,V\right)$.

We start with the following diagram:
\tikzfading[name=myfade,left color=transparent!50,right color=transparent!70,middle color=transparent!60]
\begin{equation}\label{eq:46495}
	\begin{tikzcd}[column sep=-14pt,execute at end picture={
					\foreach \Nombre in  {Dt1,Dt2,Dt3,Dt4,Dg1,Dg2,Dg3,Dg4,Df1,Df2,Df3,Df4,Df5,Df6,Df7,Df8}
						{\coordinate (\Nombre) at (\Nombre.center);}
					\fill[HotPink, path fading=myfade,fading angle=135]
					(Dg1) -- (Dg2) -- (Dg3) -- (Dg4) -- cycle;
					\fill[LightSeaGreen, path fading=myfade,fading angle=-45]
					(Dt1) -- (Dt2) -- (Dt3) -- (Dt4) -- cycle;
					\fill[MediumBlue, path fading=myfade,fading angle=-45] (Df1) -- (Df2) -- (Df3) -- (Df4) -- cycle;
					\fill[MediumBlue, path fading=myfade,fading angle=45] (Df1) -- (Df4) -- (Df8) -- (Df5) -- cycle;
					\fill[MediumBlue,path fading=myfade,fading angle=135] (Df4) -- (Df3) -- (Df7) -- (Df8) -- cycle;
				}
		]
		&                                        &                                                           & |[alias=Dt2]|\bfD_{\frakF}^{\np,?}(V)^{\gamma-1=0} \arrow[rr, "\tau-1"] \arrow[dd, hook]  &                                         & |[alias=Dt3]|\bfD_{\frakF}^{\np,?}(V)^{\gamma-\delta=0} \arrow[dd, hook]  \\
		&                                        & |[alias=Dt1]|\bfD_{\frakF}^{\np,?}(V)^{\gamma-1=0} \arrow[rr, "\tau-1",near start,crossing over]  \arrow[ru, "1-\varphi"]  &                                         & |[alias=Dt4]|\bfD_{\frakF}^{\np,?}(V)^{\gamma-\delta=0}  \arrow[ru, "1-\varphi"]  &                       \\
		& |[alias=Dg1]|\bfD_{\frakF}^{\np,?}(V)^{\tau-1=0} \arrow[rr, hook] \arrow[dd, "\gamma-1",near start] &                                                           & |[alias=Df2]|\bfD_{\frakF}^{\np,?}(V) \arrow[rr, "\tau-1",near start,rightsquigarrow] \arrow[dd, "\gamma-1",near start] &                                         & |[alias=Df3]|\bfD_{\frakF}^{\np,?}(V) \arrow[dd, "\gamma-\delta"] \\
		|[alias=Dg4]|\bfD_{\frakF}^{\np,?}(V)^{\tau-1=0} \arrow[rr, hook,crossing over] \arrow[dd, "\gamma-1"] \arrow[ru, "1-\varphi"] &                                        & |[alias=Df1]|\bfD_{\frakF}^{\np,?}(V) \arrow[rr, "\tau-1",near start,crossing over,rightsquigarrow] \arrow[ru, "1-\varphi"]  &                                         & |[alias=Df4]|\bfD_{\frakF}^{\np,?}(V) \arrow[ru, "1-\varphi"]  &                       \\
		& |[alias=Dg2]|\bfD_{\frakF}^{\np,?}(V)^{\tau^{\chi(\gamma)}-1=0} \arrow[rr, hook]                   &                                                           & |[alias=Df6]|\bfD_{\frakF}^{\np,?}(V) \arrow[rr, "\tau^{\chi(\gamma)}-1",near start,rightsquigarrow]                  &                                         & |[alias=Df7]|\bfD_{\frakF}^{\np,?}(V)                  \\
		|[alias=Dg3]|\bfD_{\frakF}^{\np,?}(V)^{\tau^{\chi(\gamma)}-1=0} \arrow[rr, hook] \arrow[ru, "1-\varphi"]                   &                                        & |[alias=Df5]|\bfD_{\frakF}^{\np,?}(V) \arrow[rr, "\tau^{\chi(\gamma)}-1",rightsquigarrow] \arrow[ru, "1-\varphi"]                  &                                         & |[alias=Df8]|\bfD_{\frakF}^{\np,?}(V) \arrow[ru, "1-\varphi"]                   &
		\arrow[from=Dt1,to=Df1,hook,crossing over]
		\arrow[from=Dt4,to=Df4,hook,crossing over]
		\arrow[from=Df1,to=Df5,"\gamma-1",near start,crossing over]
		\arrow[from=Df4,to=Df8,"\gamma-\delta",near start,crossing over]
	\end{tikzcd}
\end{equation}
\noindent where the double complex \colorbox{HotPink!40!white}{$\scrC_\gamma$} is the lateral kernel complex of the triple complex  \colorbox{MediumBlue!40!white}{$\scrC_{\frakF}\left(\bfD_{\frakF}^{\np,?}(V)\right)$}, the double complex \colorbox{LightSeaGreen!40!white}{$\scrC_\tau$} is the vertical kernel complex of $\scrC_{\frakF}\left(\bfD_{\frakF}^{\np,?}(V)\right)$, and
$$\opn{Tot}\left({\scrC_{\frakF}\left(\bfD_{\frakF}^{\np,?}(V)\right)}\right)=\opn{Kos}\left(\varphi,\Gamma_{\frakF,K},\bfD_{\frakF}^{\np,?}(V)\right)$$
is the Herr-Ribeiro complex of $\bfD_{\frakF}^{\np,?}(V)$ over $\bfA_{\FT,K}^{\np,?}$. On the other hand, \Cref{prop:48076} shows that
$$\bfD_{\frakF}^{\np,?}(V)^{\tau=1}=\bfD_{\frakF}^{\np,?}(V)^{\tau^{\chi(\gamma)}=1}=\bfD(V),\ \bfD_{\frakF}^{\np,?}(V)^{\gamma=1}=\bfD_\tau^{\np,?}(V)$$
and $\bfD_{\frakF}^{\np,?}(V)=\AKq\otimes_{\bfA_{\tau,K}}\bfD_{\tau}^{\np,?}(V)$,
indicating that
$$\scrC_\gamma = \scrC_\gamma(\bfD(V)),\ \opn{Tot}(\scrC_\gamma)=\opn{Kos}\left(\varphi,\Gamma_K,\bfD(V)\right)$$ and $$\scrC_\tau=\scrC_\tau\left(\bfD_{\tau}^{\np,?}(V)\right),\ \opn{Tot}(\scrC_\tau)=\opn{Kos}\left(\varphi,\wtilde{\Gamma}_{\tau,K},\bfD_{\tau}^{\np,?}(V)\right).$$
If we denote by $f\colon \scrC_1\lto\scrC_2$ the morphism of double complexes represented by the squiggly arrow in \Cref{eq:46495}, then
\begin{align*}
	\opn{Kos}\left(\varphi,\Gamma_{\frakF,K},\bfD_{\frakF}^{\np,?}(V)\right)=\opn{Tot}\left({\scrC_{\frakF}\left(\bfD_{\frakF}^{\np,?}(V)\right)}\right)\cong & \opn{Tot}\left(\opn{Tot}(\scrC_1)\xlongrightarrow{\opn{Tot}(f)}\opn{Tot}(\scrC_2)\right)       \\
	\cong                                                                                                                                                     & \opn{Cone}\left(\opn{Tot}(\scrC_1)\xlongrightarrow{\opn{Tot}(f)}\opn{Tot}(\scrC_2)\right)[-1].
\end{align*}
Notice that
$$\opn{Kos}\left(\varphi,\Gamma_K,\bfD(V)\right)=\opn{Tot}(\scrC_\gamma(\bfD(V)))=\opn{Tot}\left(\opn{ker}\left(\opn{Tot}(\scrC_1)\xlongrightarrow{\opn{Tot}(f)}\opn{Tot}(\scrC_2)\right)\right),$$
a classical homological algebra argument (cf. \cite[Exercise 1.5.9]{weibelIntroductionHomologicalAlgebra1994}) gives us a natural injection (i.e., the kernel complex is exact)
\begin{equation}
	\iota_{p^\infty}\colon\opn{Kos}\left(\varphi,\Gamma_K,\bfD(V)\right)\longhookrightarrow \opn{Kos}\left(\varphi,\Gamma_{\FT,K},\bfD_{\FT}^{\np,?}(V)\right).
\end{equation}
Similarly, one obtains another natural injection
\begin{equation}
	\iota_{\infty}\colon\opn{Kos}\left(\varphi,\wtilde{\Gamma}_{\tau,K},\bfD_{\tau}^{\np,?}(V)\right)\longhookrightarrow \opn{Kos}\left(\varphi,\Gamma_{\FT,K},\bfD_{\FT}^{\np,?}(V)\right).
\end{equation}

\begin{theorem}\label{thm:26607}
	The morphisms $\iota_\infty$ and $\iota_{p^\infty}$ induce quasi-isomorphisms.
\end{theorem}
\begin{proof}
	The theorem follows from comparing the cocycles. We prove $\iota_{p^\infty}$ induces quasi-isomorphism, and the same strategy applies to $\iota_\infty$.  Denote by $\kappa_{p^\infty}^*$ the quasi-isomorphism $$\opn{Kos}\left(\varphi,\Gamma_K,\bfD(V)\right)\xlongrightarrow{\cong}\rmR\Gamma_{\opn{cont}}(\scrG_K,V).$$

	By expanding the total complex, one can write the morphism $\iota_{p^\infty}$ as
	\[\begin{tikzcd}[ampersand replacement=\&,row sep=large,column sep=20pt]
			0 \& {\mathbf{D}(V)} \& {\mathbf{D}(V)^{\oplus 2}} \& {\mathbf{D}(V)} \& 0\ar[d]\ar[r]\&0 \\
			0 \& {\mathbf{D}_{\mathfrak{F}}^{\operatorname{np},?}(V)} \& {\mathbf{D}_{\mathfrak{F}}^{\operatorname{np},?}(V)^{\oplus 3}} \& {\mathbf{D}_{\mathfrak{F}}^{\operatorname{np},?}(V)^{\oplus 3}} \& {\mathbf{D}_{\mathfrak{F}}^{\operatorname{np},?}(V)} \& 0
			\arrow[from=1-1, to=1-2]
			\arrow[from=1-2, to=1-3,"f_0"]
			\arrow[from=1-2, to=2-2,"\opn{id}"]
			\arrow[from=1-3, to=1-4,"f_1"]
			\arrow[from=1-3, to=2-3,"{\begin{bsmallmatrix}1&0\\0&1\\0&0\end{bsmallmatrix}}"]
			\arrow[from=1-4, to=1-5]
			\arrow[from=1-4, to=2-4,"{\begin{bsmallmatrix}0\\0\\1\end{bsmallmatrix}}"]
			\arrow[from=2-1, to=2-2]
			\arrow[from=2-2, to=2-3,"g_0"]
			\arrow[from=2-3, to=2-4,"g_1"]
			\arrow[from=2-4, to=2-5,"g_2"]
			\arrow[from=2-5, to=2-6]
		\end{tikzcd},\]
	where $f_0=\begin{bsmallmatrix}\gamma-1\\1-\varphi\end{bsmallmatrix}$, $f_1=\begin{bsmallmatrix}\varphi-1 & \gamma-1\end{bsmallmatrix}$, $g_0=\begin{bsmallmatrix}\gamma-1\\1-\varphi\\\tau-1\end{bsmallmatrix}$, $g_1=\begin{bsmallmatrix}0&1-\tau&1-\varphi\\1-\tau^{\chi(\gamma)}&0&\gamma-\delta\\\varphi-1&\gamma-1&0\end{bsmallmatrix}$ and $g_2=\begin{bsmallmatrix}\tau^{\chi(\gamma)}-1&\delta-\gamma&\varphi-1\end{bsmallmatrix}$.

	It is easy to show that $\iota_{p^\infty}$ induces an isomorphism of $\rmH^0$. For $\rmH^1$, the isomorphism $\kappa_{p^\infty}^{*,1}$ is explicitly given by the formula (cf. \cite[Proposition I.4.1]{cherbonnierTheorieIwasawaRepresentations1999}):
	\begin{equation}\label{eq:54978}
		\begin{tikzcd}
			(x,y)\in Z^1\left(\opn{Kos}\left(\varphi,\Gamma_K,\bfD(V)\right)\right)\ar[d,maps to]\\\left(g\mapsto c_{x,y}(g)\coloneqq\frac{\gamma^n-1}{\gamma-1}x-(g-1)b, \text{ if }g\vert_{\Gamma_K}=\gamma^n\right)\in Z^1\left(\rmR\Gamma_{\opn{cont}}(\scrG,V)\right)
		\end{tikzcd}.
	\end{equation}
	where $b\in\bfA\otimes_{\Zp}V$ satisfies $(\varphi-1)b=-y$. On the other hand, the isomorphism $\kappa_{\frakF}^{*,1}$ is explicitly given by the formula (cf. \cite[Theorem 1.5 (ii)]{ribeiroExplicitFormulaHilbert2011}\footnote{Although Ribeiro established the formula exclusively for \'etale $(\varphi,\Gamma_{\frakF,K})$-modules over $\bfA_{\frakF,K}^{\ttR}$, one can still adapt his result to our variant by replacing $\bfA_{\frakF,K}^{\ttR}$ with $\bfA_{\frakF,K}^{\np,?}$ everywhere in \cite[Theorem 1.5 (ii)]{ribeiroExplicitFormulaHilbert2011}.}):
	\begin{equation}\label{eq:1990}
		\begin{tikzcd}
			(x,y,x)\in Z^1\left(\opn{Kos}\left(\varphi,\Gamma_{\frakF,K},\bfD_{\frakF}^{\np,?}(V)\right)\right)\ar[d,maps to]\\
			\left(g\mapsto c_{x,y,z}(g)\coloneqq\frac{\gamma^n-1}{\gamma-1}x-(g-1)b+\gamma^n\frac{\tau^m-1}{\tau-1}z, \text{ if }g\vert_{\Gamma_{\frakF,K}}=\gamma^n\tau^m\right)\in Z^1\left(\rmR\Gamma_{\opn{cont}}(\scrG,V)\right)
		\end{tikzcd}.
	\end{equation}
	Then it is immediate to verify the commutativity of the following diagram:
	\begin{equation}\label{eq:31224}
		\begin{tikzcd}[/tikz/column 1/.style={column sep=5pt},/tikz/column 2/.style={column sep=-5pt}]
			\rmH^1\left(\opn{Kos}\left(\varphi,\Gamma_K,\bfD(V)\right)\right) & \rmH^1\left(\scrG_K,V\right) & (x,y) & \left(g\mapsto c_{x,y}(g)\right) \\
			& \rmH^1\left(\opn{Kos}\left(\varphi,\Gamma_{\frakF,K},\bfD_{\frakF}^{\np,?}(V)\right)\right) && (x,y,0)
			\arrow[from=1-1, to=1-2, "\cong"',"\kappa_{p^\infty}^{*,1}"]
			\arrow[from=1-1, to=2-2, "\iota_{p^\infty}^{*,1}"']
			\arrow[maps to, from=1-3, to=1-4,"\eqref{eq:54978}"]
			\arrow[maps to, from=1-3, to=2-4]
			\arrow[from=2-2, to=1-2,"\cong","\kappa_{\frakF}^{*,1}"']
			\arrow[maps to, from=2-4, to=1-4,"\eqref{eq:1990}"']
		\end{tikzcd},
	\end{equation}
	Thus $\iota_{p^\infty}^{*,1}$ is an isomorphism. Similarly, one can show that the isomorphism $\kappa_{p^\infty}^{*,2}$ is given by the formula $$x\longmapsto \left((g,h)\mapsto s_g-s_{gh}+gs_h\right),$$ where $s_\sigma$ satisfies $(\varphi-1)s_\sigma=\frac{\gamma^n-1}{\gamma-1}(-x)$ if $\sigma\vert_{\Gamma_K}=\gamma^n$. By comparing this with the explicit formula of the isomorphism
	$\kappa_{\frakF}^{*,2}$ given by \cite[Proposition 1.10]{ribeiroExplicitFormulaHilbert2011}, one shows that the analogous diagram of \eqref{eq:31224} for $\rmH^2$ commutes. This shows that $\iota_{p^\infty}^{*,2}$ is also an isomorphism, and the result follows.
\end{proof}

\subsection{The cohomology of $\opn{Kos}\left(\psi^{\np},\Gamma_{\frakF,\Qp},\bfD_{\frakF}^{\np,?}(V)\right)$: attempts}
In this subsection we assume that $K=\Qp$.
\subsubsection{The $\psi$-operator}

\begin{definition}\leavevmode
	\begin{enumerate}
		\item For any $\bfE\in\left\{\bfE_{\FT,\Qp}^{\np,?},\bfE_{\FT}^{\np,?}\right\}$, by \Cref{coro:42040} we can write any element $x\in\bfE$ uniquely as $x=\sum_{i=0}^{p-1}\sum_{j=0}^{p-1}\varphi(x_{ij})u_{\Qp}^i\varepsilon^j$. Define the $\psi$-operator on $\bfE$ by
		      $$\psi_{\bfE}^{\np}\colon \bfE\to\bfE,\ x\longmapsto x_{00}.$$

		\item For any $\bfA_{\frakF}\in\left\{\bfA_{\FT,\Qp}^{\np,?},\bfA_{\FT}^{\np,?}\right\}$, by \Cref{coro:38666} we can write any element $x\in\bfA_{\frakF}$ uniquely as $x=\sum_{i=0}^{p-1}\sum_{j=0}^{p-1}\varphi(x_{ij})[u_{\Qp}]^i[\varepsilon]^j$. Define the $\psi$-operator on $\bfA_{\frakF}$ by
		      $$\psi^{\np}\colon \bfA_{\frakF}\lto \bfA_{\frakF},\ x\longmapsto x_{00}.$$
	\end{enumerate}
\end{definition}
\begin{remark}
	For $\bfE\in\left\{\bfE_{\FT,\Qp}^{\np,?},\bfE_{\FT}^{\np,?}\right\}$, the containment $\bfE\subseteq \widehat{\bfE}_{\FT}^{\np}$ implies that $\psi_{\bfE}^{\np}=\psi_{\widehat{\bfE}_{\FT}^{\np}}^{\np}\vert_{\bfE}$.
	Thus, we will simply write $\psi^{\np}$ instead of $\psi_{\bfE}^{\np}$. For the same reason, the notation $\psi^{\np}$ for the period rings of characteristic $0$ is unambiguous.
\end{remark}
It can be checked by direct calculation that
\begin{lemma}
	For every variant of $\psi^{\np}$ defined above, we have
	\begin{enumerate}
		\item $\psi^{\np}\circ \varphi=\opn{id}$;
		\item $\psi^{\np}$ commutes with the action of $\scrG_{\Qp}$.
	\end{enumerate}
\end{lemma}
\begin{proposition}\label{prop:2140}
	For every \'etale $(\varphi,\Gamma_{\FT,\Qp})$-module $D$ over $\bfA_{\FT,\Qp}^{\np,?}$, there exists a unique operator $\psi^{\np}\colon D\lto D$ satisfying
	$$\psi^{\np}(a\cdot\varphi(x))=a\cdot\psi^{\np}(x),\ \psi^{\np}(a\cdot\varphi(x))=\psi^{\np}(a)\cdot x$$
	for every $a\in \bfA_{\FT,\Qp}^{\np,?}$ and $x\in D$. Moreover, for any $x\in D$, and any positive integer $n$, one has
	\begin{equation}\label{eq:17638}
		x=\sum_{i=0}^{p^n-1}\sum_{j=0}^{p^n-1}[\varepsilon]^i\cdot u_{\Qp}^j\cdot\varphi^n\left(\psi^{\np,n}\left([\varepsilon]^{-i}\cdot u_{\Qp}^{-j}\cdot x\right)\right),
	\end{equation}
	where $\psi^{\np,n}$ is the $n$-th iterate of $\psi^{\np}$.
\end{proposition}
\begin{proof}
	See for example \cite[Corollary 5.30]{fontaineTheoryPadicGalois}. For $D=\bfD_{\FT}^{\np,?}(V)$, the $\psi$ operator on $D$ is induced by the map $\psi\otimes 1$ on $\bfA_{\FT}^{\np,?}\otimes_{\Zp}V$. \eqref{eq:17638} follows from induction on $n$.
\end{proof}
\begin{remark}\label{rmk:20098}
	Since $1,[\varepsilon],\cdots,[\varepsilon]^{p-1}$ is a basis of $\bfA$ over $\varphi(\bfA)$, the above construction implies that $\psi^{\np}$ on any \'etale $(\varphi,\Gamma_{\FT,\Qp})$-module $\bfD_{\frakF}^{\np,?}(V)$ coincides with $\psi$ when restricted to $\bfD(V)$.
\end{remark}
\subsubsection{Comparison of $\opn{Kos}\left(\varphi,\Gamma_{\frakF,\Qp},\bfD_{\frakF}^{\np,?}(V)\right)$ and $\opn{Kos}\left(\psi^{\np},\Gamma_{\frakF,\Qp},\bfD_{\frakF}^{\np,?}(V)\right)$}
By replacing $\varphi$ by $\psi^{\np}$ everywhere in $\opn{Kos}\left(\varphi,\Gamma_{\frakF,\Qp},\bfD_{\frakF}^{\np,?}(V)\right)$, one gets the $\psi^{\np}$-complex:
\begin{equation*}\opn{Kos}\left(\psi^{\np},\Gamma_{\frakF,K},\bfD_{\frakF}^{\np,?}(V)\right)\coloneqq\opn{Tot}\left(\begin{tikzcd}[column sep=tiny]%
			& \bfD_{\frakF}^{\np,?}(V)\ar[rr,"\tau-1"{xshift=0pt}]\ar[dd,"\gamma-1"{yshift=17pt}] && \bfD_{\frakF}^{\np,?}(V)\ar[dd,"\gamma-\delta"{yshift=0pt}] \\
			\bfD_{\frakF}^{\np,?}(V)\ar[ur,"1-\psi^{\np}"]\ar[dd,"\gamma-1"{yshift=0pt}]  && \bfD_{\frakF}^{\np,?}(V) \ar[ur,"1-\psi^{\np}"]\\
			& \bfD_{\frakF}^{\np,?}(V)\ar[rr,"\tau^{\chi(\gamma)}-1"{xshift=-17pt}] && \bfD_{\frakF}^{\np,?}(V) \\
			\ar[rr,"\tau^{\chi(\gamma)}-1"{xshift=0pt}]\bfD_{\frakF}^{\np,?}(V)\ar[ur,"1-\psi^{\np}"]&& \bfD_{\frakF}^{\np,?}(V)\ar[ur,"1-\psi^{\np}"]
			\arrow[from=2-3,to=4-3,"\gamma-\delta"{yshift=17pt},crossing over]
			\arrow[from=2-1, to=2-3,"\tau-1"{xshift=-17pt},crossing over]
		\end{tikzcd}\right).\end{equation*}
If one wants to imitate the roadmap used in \cite{cherbonnierTheorieIwasawaRepresentations1999} for proving the analogue of \Cref{thm:49557} in the false Tate curve extension setting, the next thing to do is to prove that $\opn{Kos}\left(\psi^{\np},\Gamma_{\frakF,K},\bfD_{\frakF}^{\np,?}(V)\right)$ is quasi-isomorphic to $\opn{Kos}\left(\varphi,\Gamma_{\frakF,K},\bfD_{\frakF}^{\np,?}(V)\right)$ (cf. \cite[Lemme I.5.2]{cherbonnierTheorieIwasawaRepresentations1999}), which consequently computes the Galois cohomology of $V$ by \Cref{thm:52100}. More specifically,
\begin{question}\label{ques:56379}
	For $V\in\repcat$, let $D=\bfD_{\frakF}^{\np,?}(V)$. Consider the following morphism of complexes:
	\begin{equation*}\begin{tikzcd}[row sep=large,column sep=40pt,/tikz/column 1/.style={column sep=-5pt},/tikz/column 2/.style={column sep=15pt},/tikz/column 6/.style={column sep=15pt},ampersand replacement=\&]
			\opn{Kos}\left( \varphi,\Gamma_{\frakF,\Qp},D \right)\colon\ar[d,"\Phi"]\&0\ar[r]\&D\ar[r]\ar[d,"\opn{id}"]\&D^{\oplus 3}\ar[r]\ar[d,"\scalebox{0.7}{$\begin{bsmallmatrix}
							1&&\\&1&\\&&-\psi^{\np}
						\end{bsmallmatrix}$}"]\&D^{\oplus 3}\ar[r]\ar[d,"\scalebox{0.7}{$\begin{bsmallmatrix}
							-\psi^{\np}&&\\&-\psi^{\np}&\\&&1
						\end{bsmallmatrix}$}"]\&D\ar[r]\ar[d,"-\psi^{\np}"]\&0\\
			\opn{Kos}\left( \psi^{\np},\Gamma_{\frakF,\Qp},D \right)\colon\&0\ar[r]\&D\ar[r]\&D^{\oplus 3}\ar[r]\&D^{\oplus 3}\ar[r]\&D\ar[r]\&0
		\end{tikzcd}.\end{equation*}
	Then does the morphism $\Phi$ induce a quasi-isomorphism?
\end{question}
\begin{lemma}\label{lem:57841}
	The answer to \Cref{ques:56379} is affirmative if and only if the following conditions hold:
	\begin{enumerate}[label=(C\arabic*)]
		\setcounter{enumi}{-1}
		\item\label{it:2549} $\bfD_{\FT}^{\np,?}(V)^{\psi^{\np}=0,\gamma=1,\tau=1}=0$;
		\item\label{it:52} For every $x,y\in\bfD_{\FT}^{\np,?}(V)^{\psi^{\np}=0}$ satisfying $\left( \tau^{\chi(\gamma)}-1 \right)x+(\delta-\gamma)y=0$, there exists $t\in\bfD_{\FT}^{\np,?}(V)^{\psi^{\np}=0}$ such that $x=(\gamma-1)t$ and $y=(\tau-1)t$;
		\item\label{it:35284} For every $z\in\bfD_{\FT}^{\np,?}(V)^{\psi^{\np}=0}$, there exist $x,y\in\bfD_{\FT}^{\np,?}(V)^{\psi^{\np}=0}$ such that $z=\left( \tau^{\chi(\gamma)}-1 \right)x+(\delta-\gamma)y$.
	\end{enumerate}
\end{lemma}
\begin{proof}
	The surjectivity of $\psi^{\np}$ implies the surjectivity of $\Phi$. The conditions correspond to the exactness of $\opn{ker}\Phi$ at degree $0,1,2$.
\end{proof}
The first condition is easy to verify:
\begin{proposition}\label{prop:16574}
	The condition \labelcref{it:2549} is true.
\end{proposition}
\begin{proof}
	\Cref{prop:48076} and \Cref{rmk:20098} imply that $\bfD_{\FT}^{\np,?}(V)^{\psi^{\np}=0,\gamma=1,\tau=1}=\bfD(V)^{\psi^{\np}=0,\gamma-1=0}$. Since $\gamma-1$ is invertible on $\bfD(V)^{\psi=0}$ (cf. \cite[Proposition I.5.1]{cherbonnierTheorieIwasawaRepresentations1999}), the result follows.
\end{proof}

On the other hand, verifying \labelcref{it:52} and \labelcref{it:35284} seems to be intractable. A closer examination of the proofs presented in \cite{herrCohomologieGaloisienneCorps1998} and \cite{cherbonnierTheorieIwasawaRepresentations1999} reveals that these calculations depend on delicate estimations of the action of the $\Gamma_K$ on the period rings $\bfE_{\Qp}, \bfA_{\Qp}$ (see, for instance \cite[Corollaire II.5.4]{cherbonnierTheorieIwasawaRepresentations1999} and \cite[Lemma 5.3.10]{colmezFontaineRingsPadic}). This relies on the trivial but critical fact that one can tell the valuation of an element in $\bfE_{\Qp}$ immediately by taking the minimal degree of nontrivial terms in its $\eta_{\Qp}$-adic expansion. However, when working with our multivariable structures, no simple way is known to estimate the valuation of general elements. This is predicted by the fact that $v^\flat$ is not discrete when restricted to these rings.

At the end of this article, we mention that the proof of \Cref{prop:60072} actually provides an upper bound of the valuation $v^\flat$ of elements in $\bfE_{\frakF,\Qp}^{\np,\circ,+}$ in terms of the \textbf{monomial valuation} $$v_{\opn{mono}}\left(\sum_{i,j\geq 0}f_{ij}u_{\Qp}^i\eta_{\Qp}^j\right)\coloneqq \min_{i,j\geq 0,f_{ij}\neq 0}i\cdot v^\flat(u_K)+j\cdot v^\flat(\eta_K),$$
as we will show in the appendix (cf. \Cref{thm:mono}). This might be helpful for establishing hardcore estimations towards \labelcref{it:52} and \labelcref{it:35284}.

\appendix
\section{Perfectoid valuation vs. monomial valuation}
The aim of this appendix is to compare the perfectoid valuation $v^\flat$ and the monomial valuation $v_{\opn{mono}}$ on $\bfE_{\frakF,\Qp}^{\np,\circ,+}$ in an explicit way, i.e.
\begin{theorem}\label{thm:mono}
	Let
	\begin{align*}
		\frakM(t)\coloneqq & \frac{p^{2\lceil t\rceil+3}}{p-1}\prod_{k=0}^{\lceil t\rceil}\left(k+\frac{p+1}{p}\right)^2-(p-1)\sum_{k=0}^{\lceil t\rceil}\left(k+1\right)p^k\prod_{l=0}^{k-1}\left(l+\frac{p+1}{p}\right) \\
		\colon             & \bfR_{\geq 0}\to\bfQ_{>0}.
	\end{align*}
	For any $f\in \bfE_{\frakF,\Qp}^{\np,\circ,+}$ with $v^\flat(f)\geq \frakM(0)$, one has
	$$v_{\opn{mono}}(f)\leq v^\flat(f)<\frakM\left(v_{\opn{mono}}(f)\right)=O\left(p^{2\cdot v_{\opn{mono}}(f)}\cdot \Gamma\left(v_{\opn{mono}}(f)+C\right)^2\right),$$
	where $C\coloneqq 3+\frac{1}{p}$ and $\Gamma(\cdot)$ is the usual Gamma function.
\end{theorem}
The proof of this theorem is based on a quantitative generalization of \Cref{prop:61330}.  To begin with, we make explicit choice of $\delta_c$ and $A_c$ in the proof of \Cref{prop:61330}:
\begin{lemma}\label{lem:56820}\leavevmode
	\begin{enumerate}
		\item For any $c\in \bfQ_{>0}$, let $\tau_c\coloneqq p^{n(c)}\in\bfZ_p\cong \wtilde{\Gamma}_{\tau,\Qp}$ with
		      $$n(c)\coloneqq\begin{cases}
				      \left\lceil\log_p\left((p-1)\cdot (c-1)\right)\right\rceil-1 & \text{if }c>\frac{p}{p-1}, \\
				      0                                                            & \text{otherwise}.
			      \end{cases}$$
		      If we set $\delta_1(c)\coloneqq \tau_c\cdot u_{\Qp}-u_{\Qp}$, then we have
		      $$\max\left(c,1+\frac{p}{p-1}\right)\leq v^\flat\left(\delta_1(c)\right)\leq \max\left(1+p\cdot(c-1),1+\frac{p}{p-1}\right).$$
		\item For any $c\in\bfQ_{>0}$, let $\gamma_c\coloneqq 1+p^{n'(c)}\in\bfZ_p^\times\cong \Gamma_{\Qp}$ with $$n'(c)\coloneqq\begin{cases}
				      \left\lceil\log_p\left((p-1)\cdot c\right)\right\rceil-1, & \text{if }c>\frac{1}{p-1}, \\
				      0,                                                        & \text{otherwise}.
			      \end{cases}$$
		      If we set $\delta_2(c)\coloneqq \gamma_c\cdot \eta_{\bfQ_p}-\eta_{\bfQ_p}$, then we have
		      $$\max\left(c,\frac{p}{p-1}\right)\leq v^\flat\left(\delta_2(c)\right)\leq \max\left(p\cdot c,\frac{p}{p-1}\right).$$
		\item The functions $v^\flat\left(\delta_1(\cdot)\right), v^\flat\left(\delta_2(\cdot)\right)\colon \bfQ_{>0}\lto \bfQ$ are non-decreasing.
	\end{enumerate}
\end{lemma}
\begin{proof}
	This is a direct consequence of the fact that for any integer $n$, one has
	$$v^\flat\left(\varepsilon^n-1\right)=\frac{p^{v_p(n)+1}}{p-1}.$$
\end{proof}
Now we restate \Cref{prop:61330} for the specific case of $K=\Qp$ and $\iota_{\theta,F}=\iota_{\np}$ as follows:
\begin{lemma}\label{lem:62817}
	For any $c\in\bfQ_{>0}$, we recursively define two sequences $$\left\{M_1^{[i]}(c)\right\}_{i\geq 0}, \left\{M_2^{[j]}(c)\right\}_{j\geq 0}\subset\bfQ_{>0}$$ by the formulae
	$$M_1^{[0]}(c)=M_2^{[0]}(c)\coloneqq c,$$
	$$M_1^{[i+1]}(c)\coloneqq \max\left\{M_1^{[i]}\left(c+(i+1)v^\flat\left(\delta_1(c)\right)\right),c+(i+1)v^\flat\left(\delta_1(c)\right),M_1^{[i]}(c)\right\}.$$
	$$M_2^{[j+1]}(c)\coloneqq \max\left\{M_2^{[j]}\left(c+(j+1)v^\flat\left(\delta_2(c)\right)\right),c+(j+1)v^\flat\left(\delta_2(c)\right),M_2^{[j]}(c)\right\}.$$
	Then
	\begin{enumerate}
		\item for any formal power series $f\in \bfF_p\bbrac{X,Y}$ with $v^\flat\left(\iota_{\np}(f)\right)\geq M_1^{[i]}(c)$, one has $v^\flat\left(\iota_{\np}\left(\partial_1^{[s]}f\right)\right)\geq c$ for $s=0,\cdots,i$;
		\item for any formal power series $f\in \bfF_p\bbrac{X,Y}$ with $v^\flat\left(\iota_{\np}(f)\right)\geq M_2^{[j]}(c)$, one has $v^\flat\left(\iota_{\np}\left(\partial_2^{[s]}f\right)\right)\geq c$ for $s=0,\cdots,j$.
	\end{enumerate}
\end{lemma}
\begin{proof}
	See \Cref{prop:61330}, \eqref{eq:27570} and \eqref{eq:27571}. Notice that when we only consider the case of $K=\Qp$ and $\iota_{\theta,F}=\iota_{\np}$, one has $B_c=0$ and $\lambda(c)=0$ in \eqref{eq:27571}.
\end{proof}
The following lemma collects some essential properties of $M_1^{[i]}(c)$ and $M_2^{[j]}(c)$ that will be used later:
\begin{lemma}\label{lem:50178}Let $i$ be a natural number. Then
	\begin{enumerate}
		\item For any $c\in\bfQ_{>0}$, one has $M_1^{[i]}(c)\ \text{(resp. $M_2^{[i]}(c)$)} \geq c$;
		\item The function $M_1^{[i]}(\cdot)\ \text{(resp. $M_2^{[j]}(\cdot)$)}\colon \bfQ_{>0}\lto \bfQ_{>0}$ is non-decreaseing.
		\item For any $c\in \bfQ_{>0}$, one has
		      $$M_1^{[i+1]}(c)=M_1^{[i]}\left(c+(i+1)v^\flat(\delta_1(c))\right)$$
		      and
		      $$M_2^{[j+1]}(c)=M_2^{[j]}\left(c+(j+1)v^\flat(\delta_2(c))\right).$$
	\end{enumerate}
\end{lemma}
\begin{proof}
	We prove the lemma for $M_1^{[i]}(c)$ by induction on $i$. The case of $i=0$ is trivial. Now assume that these assertions hold for certain $i$. Then for any $c\in\bfQ_{>0}$, we have
	$$M_1^{[i+1]}(c)\geq c+(i+1)v^\flat(\delta_1(c))\geq c,$$
	which shows the first assertion for $i+1$. As a result, the definition of $M_1^{[i+1]}(c)$ can be simplified as
	\begin{equation}\label{eq:25009}
		M_1^{[i+1]}(c)= \max\left\{M_1^{[i]}\left(c+(i+1)v^\flat\left(\delta_1(c)\right)\right),M_1^{[i]}(c)\right\}.
	\end{equation}
	For the second assertion, we fix $c_1,c_2\in\bfQ_{>0}$ with $c_1\geq c_2$.%
	By the induction hypothesis, we have
	$$M_1^{[i]}\left(c_1+(i+1)v^\flat\left(\delta_1(c_1)\right)\right)\geq M_1^{[i]}(c_1)$$
	and
	$$M_1^{[i]}\left(c_2+(i+1)v^\flat\left(\delta_1(c_2)\right)\right)\geq M_1^{[i]}(c_2).$$
	Thus it left to prove that
	$$M_1^{[i]}\left(c_1+(i+1)v^\flat\left(\delta_1(c_1)\right)\right)\geq M_1^{[i]}\left(c_2+(i+1)v^\flat\left(\delta_1(c_2)\right)\right),$$
	which is a direct consequence of \Cref{lem:56820} and the induction hypothesis.
	The last assertion follows from the second one and \Cref{eq:25009}.%

	The results for $M_2^{[j]}(c)$ can be proved similarly.
\end{proof}
The sequences $\left\{M_1^{[i]}(c)\right\}_{i\geq 0}$ and $\left\{M_2^{[j]}(c)\right\}_{j\geq 0}\subset\bfQ_{>0}$ are recursively defined and consequently hard to evaluate. The following lemma provides a more explicit upper bound of $M_1^{[i]}(c)$ and $M_2^{[j]}(c)$:
\begin{lemma}\label{lem:64038}
	Let $i,j\in\bfN$ and $c\in\bfQ_{>0}$.
	\begin{enumerate}
		\item Let
		      $$F_1^{[i]}(X)=p^i\prod_{k=0}^{i-1}\left(k+\frac{p+1}{p}\right)X-(p-1)\sum_{k=0}^{i-1}\left(k+1\right)p^k\prod_{l=0}^{k-1}\left(l+\frac{p+1}{p}\right).$$
		      Then $M_1^{[i]}(c)\leq F_1^{[i]}\left(\max\left(c,\frac{p}{p-1}\right)\right)$.
		\item Let
		      $$F_2^{[j]}(X)=p^j\prod_{k=0}^{j-1}\left(k+\frac{p+1}{p}\right)X.$$
		      Then $M_2^{[j]}(c)\leq F_2^{[j]}\left(\max\left(c,\frac{1}{p-1}\right)\right)$.
	\end{enumerate}
\end{lemma}
\begin{proof}
	We prove the first assertion by induction on $i$, the second one can be proved similarly.

	The case of $i=0$ is trivial, so we assume that the first assertion is proved for certain $i\in\bfN$ and arbitrary $c\in\bfQ_{>0}$.

	The series $\left\{F_1^{[i]}(x)\right\}$ can be easily verified to satisfy the following recursive relation:
	$$F_1^{[i]}(X)=X,F_1^{[i+1]}(X)=F_1^{[i]}\left(X+(i+1)\left(
		1+p\cdot(X-1)\right)\right),$$
	thus we have
	$$M_1^{[i+1]}(c)=M_1^{[i]}\left(c+(i+1)v^\flat(\delta_1(c))\right)\leq F_1^{[i]}\left(c+(i+1)v^\flat\left(\delta_1(c)\right)\right)$$
	When $c>\frac{p}{p-1}$, by \Cref{lem:56820} one has
	\begin{align*}
		F_1^{[i]}\left(c+(i+1)v^\flat\left(\delta_1(c)\right)\right)\leq F_1^{[i]}(c+(i+1)(1+p\cdot(c-1)))=F_1^{[i+1]}(c).
	\end{align*}

	When $c\leq \frac{p}{p-1}$, since $c+(i+1)\cdot v^\flat(\delta_1(c))>\frac{p}{p-1}$, we know that
	\begin{align*}
		F_1^{[i]}\left(c+(i+1)\cdot v^\flat(\delta_1(c))\right)\leq F_1^{[i]}\left(c+(i+1)\left(1+\frac{p}{p-1}\right)\right)\leq F_1^{[i+1]}\left(\frac{p}{p-1}\right).
	\end{align*}
	And the result follows.
\end{proof}
\begin{corollary}\label{coro:65342}
	For any $n\in\bfN$, let

	\begin{equation}\label{eq:0395}\begin{aligned}M(n)\coloneqq & F_1^{[n]}\left(F_2^{[n]}\left(\frac{p}{p-1}\right)\right)                                                                                                  \\
               =             & \frac{p^{2n+1}}{p-1}\prod_{k=0}^{n-1}\left(k+\frac{p+1}{p}\right)^2-(p-1)\sum_{k=0}^{n-1}\left(k+1\right)p^k\prod_{l=0}^{k-1}\left(l+\frac{p+1}{p}\right).\end{aligned}\end{equation}
	If $f=\sum_{i,j}f_{ij}X^iY^j\in\bfF_p\bbrac{X,Y}$ satisfies $v^\flat(\iota_{\np}(f))\geq M(n)$, then $f_{ij}=0$ for $i,j=0,1,\cdots,n$.
	In particular, one has $v_{\opn{mono}}(\iota_{\np}(f))> n$.
\end{corollary}
\begin{proof}
	The first assertion follows from \Cref{lem:64038} and the proof of \Cref{prop:60072}. In particular, it implies that $f_{ij}=0$ for any $v^\flat\left(u_{\Qp}\right)\cdot i+v^\flat\left(\eta_{\Qp}\right)\cdot j\leq n$.\, which finishes the proof.
\end{proof}

\begin{proof}[Proof of \Cref{thm:mono}]
	It is immediate to verify that the sequence $\{M(n)\}_{n\geq 1}$ is strictly increasing and tends to infinity, which can be extended to a strictly increasing function $\wtilde{M}(x)\colon [1,+\infty)\to [M(1),+\infty)$.
	Since
	$$v^\flat(f)=\wtilde{M}\left(\wtilde{M}^{-1}(v^\flat(f))\right)\geq\wtilde{M}\left(\left\lfloor\wtilde{M}^{-1}(v^\flat(f))\right\rfloor\right)=M\left(\left\lfloor\wtilde{M}^{-1}(v^\flat(f))\right\rfloor\right) ,$$
	\Cref{coro:65342} tells us that
	$v_{\opn{mono}}(f)>\left\lfloor\wtilde{M}^{-1}(v^\flat(f))\right\rfloor$, implying
	$$\frakM(v_{\opn{mono}}(f))=M(\lceil v_{\opn{mono}}(f)\rceil+1)\geq\wtilde{M}\left(v_{\opn{mono}}(f)+1\right)>v^\flat(f).$$
	For the big $O$ estimation, we just casually drop the negative terms in the definition of $\frakM(t)$ and use the fact for any $c\in\bfR_{>0}$ and $n\in\bfZ_{\geq 1}$, one has $\prod_{k=0}^{n-1} (k+c)=\frac{\Gamma(c+n)}{\Gamma(c)}$.
\end{proof}
\begin{remark}
	The condition $v^\flat(f)\geq \frakM(0)$ in \Cref{thm:mono} is used to ensure that $v^\flat(f)$ lies in the image of $\wtilde{M}$ in the proof.
\end{remark}
\begin{remark}
	It is possible to get a better bound of $v^\flat(f)$ in terms of $v_{\opn{mono}}(f)$ by modifying the definition of $M(n)$ (cf. \eqref{eq:0395}).
\end{remark}
\backmatter
\bibliographystyle{smfalpha-doi}
\bibliography{ref.bib}
\end{document}